\documentclass[11pt]{amsart}

\usepackage[pdftex]{graphicx}
\usepackage{amssymb}

\usepackage{graphicx}
\usepackage{xcolor}
\usepackage{amsmath}
\usepackage{amsfonts}
\usepackage{amsthm}
\usepackage{array}
\usepackage{amssymb}
\usepackage{varwidth}
\usepackage{subcaption}
\usepackage{array}
\usepackage{float}
\restylefloat{table}

\numberwithin{equation}{section}
\newtheorem{theorem}{Theorem}[section]

\newtheorem{proposition}[theorem]{Proposition}
\newtheorem{corollary}[theorem]{Corollary}

\newtheorem{lemma}[theorem]{Lemma}
\newtheorem*{conjecture}{Conjecture}

\theoremstyle{definition}
\newtheorem{definition}[theorem]{Definition}
\theoremstyle{remark}
\newtheorem{remark}[theorem]{Remark}

\newcommand{\ind}{\mathrm{ind}}

\newcommand{\area}{\mathrm{Area}}
\newcommand{\vol}{\mathrm{Vol}}
\newcommand{\length}{\mathrm{Length}}
\newcommand{\mD}{\mathcal{D}}
\newcommand{\mC}{\mathcal{C}}
\newcommand{\mR}{\mathcal{R}}

\newcommand{\del}{\partial}
\newcommand{\ip}[1]{\langle #1 \rangle}
\newcommand{\Span}{\mathrm{Span}}
\DeclareMathOperator*{\esssup}{ess\;sup}

\DeclareMathOperator{\sgn}{sgn}

\begin{document}
\title[Extremal metrics and  $n$-harmonic maps]{Laplace and Steklov extremal metrics via $n$-harmonic maps}
\begin{abstract}
We present a unified description of extremal metrics for
the Laplace and Steklov eigenvalues on manifolds of arbitrary dimension
using the notion of $n$-harmonic maps. Our approach extends the well-known 
results linking extremal metrics for eigenvalues on surfaces with minimal immersions 
and harmonic maps. In the process, we uncover two previously unknown features of 
the Steklov eigenvalues. First, we show that in higher dimensions there is a unique normalization
involving both the volume of the boundary and of the manifold itself, which leads to meaningful
extremal eigenvalue problems. Second, we observe that the critical points of the eigenvalue functionals in 
a fixed conformal class have a natural geometric interpretation provided one considers the Steklov 
problem with a density. As an example, we construct a family of free boundary harmonic annuli in 
the three-dimensional ball and conjecture that they correspond to metrics maximizing the first 
Steklov eigenvalue in  their respective conformal classes.
\end{abstract}

\author[M. Karpukhin]{Mikhail Karpukhin}
\address{Mathematics 253-37, Caltech, Pasadena, CA 91125, USA
}
\email{mikhailk@caltech.edu}
\author[A. M\'etras]{Antoine M\'etras}
\address{D\'epartement de math\'ematiques et de statistique, Universit\'e de Montr\'eal, CP 6128 succ Centre-Ville, Montr\'eal, QC H3C 3J7, Canada
}
\email{antoine.metras@umontreal.ca}
\maketitle

\section{Introduction}
\subsection{Geometric optimization of eigenvalues}
Let $(\Sigma^n,g)$ be a compact connected Riemannian $n$-dimensional manifold with non-empty boundary. The Steklov eigenvalues $\sigma_k(\Sigma^n,g)$ are defined to be the numbers $\sigma$ for which the following equation
\begin{equation*}
  \begin{cases}
    \Delta_g u = 0 & \text{in } \Sigma^n, \\
    \partial_\nu u = \sigma u & \text{on } \partial \Sigma^n
  \end{cases}
\end{equation*}
has a non-trivial solution. They form a sequence
$$
0=\sigma_0(\Sigma^n,g)<\sigma_1(\Sigma^n,g)\leqslant\ldots\nearrow\infty,
$$
where the numbers are repeated according to their multiplicity. For known results and open questions on the Steklov problem, see the review article \cite{GP}.

In the case of surfaces ($n=2$) there is a vast literature on the subject of sharp isoperimetric inequalities for Steklov eigenvalues. Namely, given a surface $\Sigma^2$, one looks for the sharp upper bounds on {\em normalized eigenvalues}
$$
\bar\sigma_k(\Sigma^2,g) = \sigma_k(\Sigma^2,g)\length(\partial\Sigma^2,g).
$$
The problem is of particular interest due to the connection with the theory of free boundary minimal surfaces established by Fraser and Schoen~\cite{FS2, FS3}. They showed that {\em extremal metrics} for $\bar\sigma_k$ (i.~e. critical points of $\bar\sigma_k(\Sigma^2,g)$ as a function of $g$, see Definition~\ref{extremal:def}) correspond to free boundary minimal surfaces in Euclidean balls. Furthermore, in~\cite{FS2} this connection is used to obtain sharp upper bounds for $\bar\sigma_1$ on the annulus and the M\"obius strip. In recent years there have been further developments in the study of the interactions between bounds for Steklov eigenvalues and free boundary minimal surfaces, see e.g.~\cite{MP, MP2, Pet3, GL, GKL, KS}, as well as the study of free boundary minimal submanifolds for general ambient manifolds, see the recent survey~\cite{Li}.

The theory of Steklov eigenvalues on surfaces with boundary is in many ways parallel to the theory of Laplace eigenvalues on closed surfaces. Recall that for $(M^n, g)$ a closed, connected Riemannian $n$-dimensional manifold, the Laplace eigenvalues $\lambda_k(M^n,g)$ are the numbers $\lambda$ such that the equation
\begin{equation*}
    \Delta_g u = \lambda u
\end{equation*}
has a non-trivial solution. They form an increasing sequence 
\begin{equation*}
    0 = \lambda_0(M^n,g) < \lambda_1(M^n,g) \leqslant \ldots \nearrow \infty,
\end{equation*} 
where the numbers are repeated according to their multiplicity. Finally, the {\em normalized eigenvalues} are defined as 
$$
\bar\lambda_k(M^n,g) = \lambda_k(M^n,g)\vol(M^n,g)^{\frac{2}{n}}.
$$
For $n=2$ there is a strong analogy between sharp upper bounds for $\bar\lambda_k$ and sharp upper bounds for $\bar\sigma_k$. In fact, the results of Fraser and Schoen were motivated by the results of Nadirashvili~\cite{NadirashviliTorus}, who established the correspondence between $\bar\lambda_k$-extremal metrics  and closed minimal surfaces in spheres. Similarly, one could consider {\em conformally extremal metrics}, i.e. critical points of $\bar\sigma_k$ and $\bar\lambda_k$ restricted to a fixed conformal class of metrics. The analogy persists: $\bar\lambda_k$-conformally extremal metrics correspond to harmonic maps to spheres, whereas $\bar\sigma_k$-conformally extremal metrics give rise to free boundary harmonic maps to balls. Table~\ref{surfaces:table} contains a brief summary of these results, see Section~\ref{results:sec} for a more detailed account.

The remarkable feature of this connection is that the extremal metric is  determined by a geometric object that is defined independently of the metric. Indeed, minimal surfaces only depend on the metric in the ambient manifold, and for $n=2$ the conformal invariance of the Dirichlet energy implies that the harmonic maps depend only on the conformal structure of the surface. As a result, questions about extremal metrics can be reformulated on the language of minimal surfaces and harmonic maps, see~\cite{EGJ, JNP, KNPP, KarpukhinRP2, KS, MS, NadirashviliTorus, Pet1, Pet2} for many applications of this approach to sharp eigenvalue bounds.

The goal of the present paper is to extend this analogy to higher dimensions $n\geqslant 3$. Surprisingly, there are only a few results on the subject of $\bar\lambda_k$-(conformally) extremal metrics~\cite{CES, ESI1,ESI2, M} and extremal metrics for Steklov eigenvalues were not studied at all. In~\cite{ESI2} the authors show that $\bar\lambda_k$-extremal metrics correspond to minimal submanifolds of the sphere and that $\bar\lambda_k$-conformally extremal metrics correspond to harmonic maps with additional properties. While the former is consistent with the case of surfaces, the latter is not since the Dirichlet energy is no longer conformally invariant for $n\geqslant 3$. Later, it was noted in~\cite{M} that the conditions on the map in~\cite{ESI2} are equivalent to the property of being $n$-harmonic, which is a conformally invariant property. In the present paper we show that $n$-harmonic maps naturally appear in the study of conformally extremal metrics for Steklov eigenvalues. In particular, the formalism of $n$-harmonic maps allows one to extend the analogy between the Laplacian and Steklov problems to all dimensions in a unified way. Our main results are summarized in Table~\ref{highd:table}. The results for Steklov eigenvalues in the third column are the main contribution of the present paper, see Section~\ref{results:sec} for a more detailed account. 

While our description of $\bar\lambda_k$-(conformally) extremal metrics presented in Theorems~\ref{Laplacianconf:thm},~\ref{Laplace-extremal:thm} is mostly a reformulation of the one presented in~\cite{ESI2, M}, the situation is quite different in the Steklov case. We encounter several challenges related to previously undiscovered features of the higher-dimensional Steklov problem. They are outlined below.

\subsection{Normalization of Steklov eigenvalues} 
The first question is how to define $\bar\sigma_k$, i.e. what is the natural normalization for Steklov eigenvalues when $n\geqslant 3$. Up until now, in most papers the preferred normalization of Steklov eigenvalues is either by $\vol(\partial\Sigma^n,g)$ or by $\vol(\Sigma^n,g)$. For example, the papers~\cite{CEG1, FS1, FS4, Has} contain some non-sharp upper bounds, whereas in~\cite{CG, CEG2} the authors discuss flexibility of the Steklov spectrum, i.e. they construct examples showing that it is impossible to obtain upper bounds in various contexts. 

We propose the following definition,
\begin{equation}
\label{barsigmak:eq}
\bar\sigma_k(\Sigma^n,g) = \sigma_k(\Sigma^n,g)\vol(\partial\Sigma^n,g)\vol(\Sigma^n,g)^\frac{2-n}{n}.
\end{equation}
This normalization appears as an intermediate step in all the papers on upper bounds in higher dimensions mentioned above. For example, it appears in~\cite{CEG1, Has} due to the fact that the authors are using the approach of metric-measure spaces developed in~\cite{Korevaar, GNY}. For any Radon measure $\mu$ on  $(\Sigma^n,g)$ one defines its eigenvalues via the variational characterization
\begin{equation}
\label{measureev:eq}
\lambda_k(\Sigma^n,g,\mu) = \inf_{F_{k+1}}\sup_{u\in F_{k+1}\setminus \{0\}}\frac{\displaystyle\int |d u|_g^2\,dv_g}{\displaystyle\int u^2\,d\mu},
\end{equation}
where $F_{k+1}$ varies over $(k+1)$-dimensional subspaces of $C^\infty(M)$ that remain $(k+1)$-dimensional in $L^2(\mu)$, see e.g.~\cite{Kokarev, GKL, KS}. These eigenvalues have the following scaling property $\lambda_k(\Sigma^n,tg,s\mu) =s^{-1} t^{\frac{n-2}{2}}\lambda_k(\Sigma^n,g,\mu)$. Therefore, the natural normalized quantity is 
\begin{equation}
\label{barlambdakmu:eq}
\bar\lambda_k(\Sigma^n,g,\mu) = \lambda_k(\Sigma^n,g,\mu)\mu(\Sigma^n)\vol(\Sigma^n,g)^\frac{2-n}{n}.
\end{equation}
If $\mu$ equals the boundary measure of $\Sigma^n$, then $\lambda_k(\Sigma^n,g,\mu)$ coincides with the Steklov eigenvalues and one recovers normalization~\eqref{barsigmak:eq} from~\eqref{barlambdakmu:eq}.

Furthermore, we show that for any other normalization by powers of $\vol(\partial \Sigma^n,g)$ and $\vol(\Sigma^n,g)$ the normalized eigenvalue functional does not have any smooth critical points. This fact explains some of the flexibility results mentioned above. All the observations made in this section point us towards the fact that~\eqref{barsigmak:eq} is the natural normalization for the Steklov problem, see also~\cite{GKL} for some additional arguments. One of the main goals of the present paper is to convince the reader that it is indeed the case.

\subsection{Steklov problem with a density.} 
The second issue we encounter is that even with the correct normalization the $\bar\sigma_k$-conformally extremal metrics do not quite correspond to free boundary $n$-harmonic maps. In order to remedy the situation we consider the Steklov problem with a density.

Let $0<\rho\in C^\infty(\partial \Sigma^n)$ be a positive smooth function. The eigenvalues $\sigma_k(\Sigma^n,g,\rho)$ are defined to be the numbers $\sigma$ such that the following equation 

\begin{equation*}
  \begin{cases}
    \Delta_g u = 0 & \text{in } \Sigma^n, \\
    \partial_\nu u = \sigma \rho u & \text{on } \partial \Sigma^n
  \end{cases}
\end{equation*}
has a non-trivial solution. The eigenvalues $\sigma_k(\Sigma^n,g,\rho)$ are a direct generalization of $\sigma_k(\Sigma^n,g,\rho)$ with the same properties. Furthermore, $\sigma_k(\Sigma^n,g,\rho) = \lambda_k(\Sigma^n,g,\rho \,dA_g)$, so the natural normalization is
\begin{equation}
\label{barsigmakrho:eq}
\bar\sigma_k(\Sigma^n,g,\rho) = \sigma_k(\Sigma^n,g,\rho)\left(\int_{\partial\Sigma^n}\rho\,dA_g\right)\vol(\Sigma^n,g)^\frac{2-n}{n}.
\end{equation}

We prove that critical points of $\bar\sigma_k(\Sigma^n,g,\rho)$ in the product space $\mC\times C^\infty(\Sigma^n)$ correspond to free boundary $n$-harmonic maps to the unit ball, see Theorem~\ref{Steklovconf:thm} below for the exact statement. For $n=2$ the conformal invariance of the Laplacian implies that for any $0\leqslant f\in C^\infty(\Sigma^2)$ one has $\sigma_k(\Sigma^2,g,f|_{\partial\Sigma^2}) = \sigma_k(\Sigma^2,f^2g)$. In this case we recover the correspondence with free boundary harmonic maps to the unit ball for Steklov eigenvalues without density established in~\cite{FS3}, see Section~\ref{Steklovsurfaces:sec} for a more detailed discussion. However, the lack of the conformal invariance in higher dimensions makes it necessary to consider the problem with density. 

\subsection{Free boundary harmonic annuli in $\mathbb{B}^3$} In the final section we construct and study a family of rotationally symmetric free boundary harmonic maps from an annulus $\mathbb{A}_T=[0,T]\times \mathbb{S}^1$ to a three-dimensional ball $\mathbb{B}^3$. Our approach is very similar to Fraser-Schoen's analysis of rotationally symmetric metrics on annuli in~\cite{FS1}. We show that for any $T\geqslant T_1$ there exists a rotationally symmetric free boundary harmonic map $\Psi_T\colon\mathbb{A}_T\to\mathbb{B}^3$ corresponding to a $\bar\sigma_1$-conformally extremal metric. The image of $\Psi_T$ is a piece of a stretched catenoid, see Figure~\ref{fig:examples}. In Section~\ref{fbha:sec} we conjecture that these metrics are in fact $\bar\sigma_1$-maximizers in their conformal class and provide some evidence for this conjecture. Finally, we show that there are values of $T$ such that the density $\rho$ corresponding to $\Psi_T$ is not identically constant. As a result, there exist critical pairs $(g,\rho)$ with non-constant $\rho$, i.e. it is necessary to introduce the density function in order to have a correspondence with free boundary harmonic maps.

\subsection*{Plan of the paper} The paper is organized in the following way. In Section~\ref{results:sec} we provide the statements of the main results and recall the necessary background on $n$-harmonic maps and minimal immersions. In Section~\ref{alg_extremal:sec} the definition of the metric extremal for the eigenvalue functional is introduced. We obtain the necessary and sufficient conditions of extremality in terms of algebraic relations on the corresponding eigenfunctions. These relations are then used in Section~\ref{geometricextremal:sec} to establish the correspondence between extremal metrics, $n$-harmonic maps and minimal immersions. Finally, Section~\ref{Annuli:sec} contains the construction of the free boundary harmonic annuli with non-trivial boundary densities.

\subsection*{Acknowledgements} The authors would like to thank Daniel Stern for fruitful discussions and Iosif Polterovich for invaluable remarks on the preliminary versions of the manuscript.  
Antoine Métras's work is part of a PhD thesis under the supervision of Iosif Polterovich and was supported by the Fonds de recherche du Québec – Nature et technologies (FRQNT). Mikhail Karpukhin is partially supported by US National Science Foundation grant DMS-1363432.

\newpage


\begin{table}[h!]
	\centering
	
	{\renewcommand{\arraystretch}{2}
	
	\begin{tabular}{  m{0.3\textwidth} | m{0.35\textwidth}| m{0.35\textwidth}  }
	 		       & Laplacian $\bar\lambda_k(M,g)$  & Steklov $\bar\sigma_k(\Sigma,g)$ \\
			       \hline
	Normalization  & $ \lambda_k(M,g)\area(M,g)$ & $\sigma_k(\Sigma,g)\length(\partial \Sigma,g)$ \\
	\hline
	Extremal metrics in the conformal class $\mC$ correspond to & Harmonic maps \newline $\Phi\colon (M,\mC)\to\mathbb{S}^{m-1}$ & Free boundary harmonic maps $\hat\Psi\colon(\Sigma,\mC)\to\mathbb{B}^{m}$\\
	\hline
	which are extremal points of the conformally invariant functional & $ E(\Phi) = \frac{1}{2}\displaystyle\int |d\Phi|_g^2\,dv_g $ & $E(\hat\Psi) = \frac{1}{2}\int|d\hat\Psi|_g^2\,dv_g$ under the constraint that\newline $\Psi = \hat\Psi|_{\partial \Sigma}\colon\partial \Sigma\to\mathbb{S}^{m-1}$ \\
	\hline
	and, as a result, the solutions of & $\Delta_g\Phi = |d\Phi|_g^2\Phi$ & $\mD_g\Psi = |\mD_g\Psi|\Psi$ or\newline $\partial_\nu\hat\Psi = |\partial_\nu\hat\Psi|\hat\Psi$  \\
	\hline
	\multicolumn{3}{l}{	}

	\varwidth{\linewidth}
	Conversely, any corresponding geometric object gives rise 
	 to the extremal metric
	 \endvarwidth\vspace{1mm}
	 \\
	\hline
	One defines the associated conformally covariant operator & $L_{g,\Phi} = \Delta_g-|d\Phi|_g^2$ & $L^{\mD}_{g,\Psi} = \mD_g - |\mD_g\Psi|$ \\
	\hline
	and the spectral index & $\ind_S(\Phi) = \ind(L_{g,\Phi})$ & $\ind_S(\Psi) = \ind(L^\mD_{g,\Psi})$ \\
	\hline
	The corresponding metric is given by & $g_\Phi = |d\Phi|_g^2\,g$ & any $h$ with \newline $h|_{\partial \Sigma} = |\mD_g\Psi| \,g|_{\partial \Sigma}$ \\
	\hline
	which is extremal for the functional & $\bar\lambda_{\ind_S(\Phi)}(M,g)$ & $\bar\sigma_{\ind_S(\Psi)}(\Sigma,g)$ \\
	\hline
	with the value & $\bar\lambda_{\ind_S(\Phi)}(M,g_\Phi) = 2E(\Phi)$ & $\bar\sigma_{\ind_S(\Psi)}(\Sigma,h) = 2E(\hat\Psi)$\\
	\hline
	Extremal metrics in the space of all metrics correspond to the same geometric object with an additional property of being conformal, i.e. to & Minimal immersions to $\mathbb{S}^{m-1}$ & Free boundary minimal immersions to $\mathbb{B}^{m}$\\
	\hline
	\end{tabular}
	}
	
	\caption{Classical theory of extremal metrics on surfaces.}
	\label{surfaces:table}
	\end{table}

\newpage

\begin{table}[H]
	\centering
	{\renewcommand{\arraystretch}{1.8}
	\begin{tabular}{  m{0.3\textwidth} | m{0.35\textwidth}| m{0.35\textwidth}  }
	 		       & Laplacian $\bar\lambda_k(M,g)$  & Steklov $\bar\sigma_k(\Sigma,g,\rho)$ \\
			       \hline
	Normalization  & $ \lambda_k(M^n,g)\vol(M^n,g)^\frac{2}{n}$ & $\sigma_k(\Sigma^n,g)\|\rho\|_{L^1}\vol(\Sigma^n,g)^{\frac{2-n}{n}}$ \\
	\hline
	Extremal metrics in the conformal class $\mC$ correspond to & $n$-Harmonic maps \newline $\Phi^n\colon (M^n,\mC)\to\mathbb{S}^{m-1}$ & Free boundary $n$-harmonic maps $\hat\Psi^n\colon(\Sigma,\mC)\to\mathbb{B}^{m}$\\
	\hline
	which are extremal points of the conformally invariant functional & $ E^n(\Phi^n) = \frac{1}{n}\displaystyle\int |d\Phi^n|_g^n\,dv_g $ & $E^n(\hat\Psi^n) = \frac{1}{n}\int |d\hat\Psi^n|_g^n\,dv_g$ under the constraint that $\Psi^n : = \hat\Psi^n|_{\partial \Sigma^n}\colon\partial \Sigma^n\to\mathbb{S}^{m-1}$ \\
	\hline
	and, as a result, the solutions of & $\delta_g(|d\Phi|_g^{n-2}d\Phi) = |d\Phi|_g^n\Phi$ & $\partial_\nu\hat\Psi^n = |\partial_\nu\hat\Psi^n|\hat\Psi^n$ or\newline $\mD^n_{g,\Psi^n}\Psi^n = |\mD^n_{g,\Psi^n}\Psi^n|\Psi^n$, where $\mD^n_{g,\Psi^n}$ is the Dirichlet-to-Neumann operator associated with $P(u) = \delta_g(|d\hat\Psi^n|^{n-2}_g du)$  \\
	\hline
	\multicolumn{3}{l}{	}

	\varwidth{\linewidth}
	Conversely, any corresponding geometric object with nowhere vanishing differential gives rise 
	 to the extremal metric
	 \endvarwidth\vspace{1mm}
	 \\
	\hline
	One defines the associated conformally covariant operator & $L_{g,\Phi^n}(u) =$ \newline$ \delta_g(|d\Phi^n|_g^{n-2}d u)-|d\Phi^n|_g^n u$ & $L^{\mD}_{g,\Psi^n}(u) =$\newline$ \mD^n_{g,\Psi^n}u - |\mD^n_{g,\Psi^n}\Psi^n|u$ \\
	\hline
	and the spectral index & $\ind_S(\Phi^n) = \ind(L_{g,\Phi^n})$ & $\ind_S(\Psi^n) = \ind(L^\mD_{g,\Psi^n})$ \\
	\hline
	The corresponding metric (density) is given by & $g_{\Phi^n} = |d\Phi^n|_g^2g$ & $g_{\Psi^n} = |d\hat\Psi^n|_g^2g$ \newline $\rho_{\Psi^n} =  |\partial_{\nu_g}\hat\Psi|/|d\hat\Psi^n|_g$ \\
	\hline
	which is extremal for the functional & $\bar\lambda_{\ind_S(\Phi^n)}(M^n,g)$ & $\bar\sigma_{\ind_S(\Psi^n)}(\Sigma^n,g,\rho)$ \\
	\hline
	with the value & $\bar\lambda_{\ind_S(\Phi^n)}(M^n,g_{\Phi^n})$ $=nE^n(\Phi^n)$ & $\bar\sigma_{\ind_S(\Psi^n)}(\Sigma^n,g_{\Psi^n},\rho_{\Psi^n})$\newline $= nE^n(\hat\Psi^n)$\\
	\hline
	Extremal metrics in the space of all metrics correspond to the same geometric object with an additional property of being conformal, i.e. to & Minimal immersions to $\mathbb{S}^{m-1}$ & Free boundary minimal immersions to $\mathbb{B}^{m}$. \newline Additionally, extremal densities are always constant.\\
	\hline
	\end{tabular}
	}
	\caption{The theory of extremal metrics in dimension $n\geqslant 3$. 
	}
	\label{highd:table}
	\end{table}

\pagebreak

\section{Main results} 
\label{results:sec}
In this section we provide the detailed account of the main results outlined in the introduction.

\subsection{Extremal metrics in the conformal class: Laplacian} We provide a geometric characterization of extremal metrics in a fixed conformal class. This is well-known for surfaces~\cite{NadirashviliTorus, ESI2, FS3}. For Laplacian eigenvalues in higher dimensions it has previously appeared in a slightly different form in~\cite{ESI2, M}, for Steklov eigenvalues our results are new. The discussion below is summarized in Tables~\ref{surfaces:table},~\ref{highd:table}.

We first recall some background information on $n$-harmonic maps.
\begin{definition}
Let $(N,g)$, $(Q,h)$ be Riemannian manifolds. Then the map $\Phi\colon (N,g)\to(Q,h)$, $\Phi\in W^{1,p}(N,Q)$ is called $p$-harmonic if it is a critical point of the $p$-energy functional
$$
E^p_g(\Phi) = \frac{1}{p}\int_N |d\Phi|^p_{g,h}\,dv_g. 
$$
If $p=2$, then $p$-harmonic maps are referred to as simply harmonic.
\end{definition}
In the present paper we only consider a case $n=\dim N = p$. In this case, the energy functional $E^n_g$ is conformally invariant, i.e. for any $\Phi\in W^{1,n}(N,Q)$ one has $E^n_g(\Phi) = E^n_{e^{2\omega}g}(\Phi)$. Thus, the property of being $n$-harmonic only depends on the conformal class $[g]$ of the metric $g$. 

Assume $N=M^n$ is an $n$-dimensional closed manifold and $(Q,h)$ is a unit sphere $\mathbb{S}^{m-1}\subset\mathbb{R}^{m}$ with the standard metric. The regularity theory for such $n$-harmonic maps asserts that they are always $C^{1,\alpha}$ for some $\alpha>0$~\cite[Corollary 12]{Takeuchi} and $C^\infty$ if $n=2$ or if $d\Phi\ne 0$.
A straightforward computation shows that $\Phi\colon (M^n,g)\to \mathbb{S}^{m-1}$ is $n$-harmonic iff it is a weak solution of
\begin{equation}
\label{n-harmonic:eq}
\delta_g(|d\Phi|_g^{n-2}d\Phi) = |d\Phi|^n_g\Phi,
\end{equation}    
where $\delta_g$ is the dual of $d$. 

Assume further that $\Phi$ is non-degenerate, i.e. $d\Phi$ does not vanish on $M^n$. To each such $\Phi$ one associates a Schr\"odinger operator
$$
L_{g,\Phi}(u) = \delta_g(|d\Phi|_g^{n-2}du) - |d\Phi|^n_gu,
$$
which can be seen to be conformally covariant, $L_{e^{2\omega}g,\Phi} = e^{-n\omega}L_{g,\Phi}$. In particular, the index of $L_{g,\Phi}$, i.e. the number of negative eigenvalues, is independent of the choice of $g$ in the conformal class.

\begin{definition}
Let $\Phi\colon (M^n,\mC)\to\mathbb{S}^{m-1}$ be a non-degenerate $n$-harmonic map. The {\em spectral index} $\ind_S(\Phi)$ is defined to be the index of the operator $L_{g,\Phi}$ for some (any) metric $g\in\mC$.
\end{definition}

Note that by the equation of the $n$-harmonic map one has  $L_{g,\Phi}(\Phi)=0$, i.e. the components of $\Phi$ are in the kernel of $L_{g,\Phi}$. Since $\Phi$ is non-degenerate, we can define a smooth metric $g_\Phi$ by the formula $g_{\Phi} = \frac{1}{n}|d\Phi|_g^2g$. Then one has 
$$
L_{\Phi,g_{\Phi}} = n^{\frac{n-2}{2}}(\Delta_{g_{\Phi}} - n),
$$
i.e. the components of $\Phi$ are eigenfunctions of $\Delta_{g_{\Phi}}$ with the eigenvalue $\lambda_k=n$. Furthermore, $\ind_S(\Phi)$ is the smallest $k$ such that $\lambda_k(M^n,g_\Phi)=n$ is satisfied.

We are now in position to state a geometric characterization of extremal metrics for Laplacian eigenvalues. We set 
$$
\bar\lambda_k(M^n,g) = \lambda_k(M^n,g)\vol(M^n,g)^{\frac{2}{n}}.
$$
Traditionally a metric $g$ on $M^n$ is called {\em $\bar\lambda_k$-conformally extremal} if it is a {\em critical point} of the functional $g\mapsto\bar\lambda_k(M^n,g)$ in the conformal class $[g]$, see Definition~\ref{extremal:def} for the precise formulation.

\begin{theorem}
\label{Laplacianconf:thm}
Let $M^n$ be an $n$-dimensional closed manifold and $\mC$ be a conformal class on $M^n$. Suppose that the smooth metric $g\in\mC$ is $\bar\lambda_k$-conformally extremal. Then there exists a non-degenerate $n$-harmonic map $\Phi\colon (M,\mC)\to\mathbb{S}^{m-1}$ such that $g=\alpha g_\Phi$ for some $\alpha>0$ and $\bar\lambda_k(M^n,g_\Phi) = n$. In particular, $\ind_S(\Phi)\leqslant k$.

Conversely, let $\Phi\colon (M,\mC)\to\mathbb{S}^{m-1}$ be a non-degenerate $n$-harmonic map. Then the metric $g_\Phi$ is $\bar\lambda_{\ind_S(\Phi)}$-conformally extremal.
\end{theorem}

The characterization of $\bar\lambda_k$-conformally extremal metrics (in a different form) appeared in~\cite{ESI2}. The connection to $n$-harmonic maps was first observed in~\cite{M}.

\subsection{Extremal metrics in the conformal class: Steklov} Let $\Sigma^n$ be a manifold with non-empty boundary. An $n$-harmonic map $\hat\Psi\colon\Sigma^n\to\mathbb{B}^{m}$ is called {\em free boundary $n$-harmonic map} if it is proper (i.e. $\hat\Psi(x)\in \mathbb{S}^{m-1}$ iff $x\in \partial \Sigma^n$) and $\hat\Psi(\Sigma^n)\perp\mathbb{S}^{m-1}$. It is easy to see that such $\hat\Psi$ is a weak solution of the equation
\begin{equation*}
  \begin{cases}
    \delta_g(|d\hat\Psi|_g^{n-2}d\hat\Psi) = 0 & \text{in } \Sigma^n, \\
    \partial_{\nu_g} \hat\Psi = |\partial_{\nu_g}\hat\Psi| \hat\Psi & \text{on } \partial \Sigma^n.
  \end{cases}
\end{equation*}
We reserve the notation $\Psi$ for the restriction of $\hat\Psi$ to the boundary, $\Psi\colon \partial\Sigma^n\to\mathbb{S}^{m-1}$. The map $\Psi$ completely determines $\hat\Psi$.

Assume that $\hat\Psi$ is a non-degenerate. To each such $\hat\Psi$ one associates a Dirichlet-to-Neumann operator $L_{g,\hat\Psi}^\mD\colon C^\infty(\partial \Sigma^n)\to C^\infty(\partial \Sigma^n)$ as follows. For each $u\in C^\infty(\partial\Sigma^n)$ one first extends it to $\hat u\in C^\infty(\Sigma^n)$ as
\begin{equation}
\label{psiextension:eq}
  \begin{cases}
    \delta_g(|d\hat\Psi|_g^{n-2}d\hat u) = 0 & \text{in } \Sigma^n, \\
    \hat u = u & \text{on } \partial \Sigma^n.
  \end{cases}
\end{equation}
If $\mD_{g,\Psi}$ denotes the corresponding Dirichlet-to-Neumann operator $u\mapsto \partial_{\nu_g} \hat u$, then one defines
$$
L_{g,\hat\Psi}^\mD(u) = \mD_{g,\Psi}u - |\mD_{g,\Psi}\Psi|u = \partial_{\nu_g}\hat u - |\partial_{\nu_g}\hat\Psi|u.
$$
The operator $L_{g,\hat\Psi}^\mD$ is conformally covariant, $L_{e^{2\omega}g,\hat\Psi}^\mD = e^{-\omega}L_{g,\hat\Psi}^\mD$, therefore one can define the spectral index of $\hat\Psi$ as the index of $L_{g,\hat\Psi}^\mD$.

\begin{definition}
Let $\hat\Psi\colon (\Sigma^n,\mC)\to\mathbb{B}^{m+1}$ be a non-degenerate free boundary $n$-harmonic map. The {\em spectral index} $\ind_S(\hat\Psi)$ is defined to be the index of the operator $L^\mD_{g,\hat\Psi}$ for some (any) metric $g\in \mC$.
\end{definition}

By the equation of the free boundary $n$-harmonic map one has $L_{g,\hat\Psi}^\mD(\Psi)=0$. Since $\hat\Psi$ is non-degenerate we can define a smooth metric $g_{\hat\Psi} = \frac{1}{n}|d\hat\Psi|^2_gg$ on $\Sigma^n$. The extension $\hat u$ of equation~\eqref{psiextension:eq} is the harmonic extension in metric $g_{\hat\Psi}$. Furthermore, setting $\rho_{\hat\Psi} = |\partial_{\nu_{g_{\hat\Psi}}} \hat\Psi| = |\mD_{g_{\hat\Psi},\Psi}\hat\Psi| = \frac{\sqrt{n}|\mD_{g,\Psi}\hat\Psi|}{|d\hat\Psi|_g}\in C^\infty(\partial\Sigma^n)$, we show in Lemma~\ref{density:lemma} that $\rho_{\hat\Psi}>0$. Thus,
$$
L_{g_{\hat\Psi},\hat\Psi} = \mD_{g_{\hat\Psi}} - \rho_{\hat\Psi}
$$
is the classical Dirichlet-to-Neumann map with density. In particular, the components of $\hat\Psi$ are $\sigma_k(\Sigma^n,g_{\hat\Psi},\rho_{\hat\Psi})$ eigenfunctions with eigenvalue $\sigma_k=1$. The smallest $k$ such that 
$\sigma_k(\Sigma^n,g_{\hat\Psi},\rho_{\hat\Psi}) =1$ is the spectral index $\ind_S(\hat\Psi)$.

We can now formulate the geometric characterization of extremal metrics. Recall the normalization
$$
\bar\sigma_k(\Sigma^n,g,\rho) = \sigma_k(M,g,\rho)\vol(\Sigma^n,g)^\frac{2-n}{n}||\rho||_{L^1(\partial\Sigma^n,g)}.
$$ 
The pair $(g,\rho)$ is called $\bar\sigma$-conformally extremal if it is a critical point of the functional $\bar\sigma_k$ in $[g]\times C_{>0}^\infty(\partial\Sigma^n)$.

\begin{theorem}
\label{Steklovconf:thm}
Let $\Sigma^n$ be an $n$-dimensional compact manifold with boundary of dimension $n\geqslant 3$ and $\mC$ be a conformal class on $\Sigma^n$. Suppose that the pair $(g,\rho)\in\mC\times C_{>0}^\infty(\partial\Sigma^n)$ is $\bar\sigma_k$-conformally extremal. Then there exists a non-degenerate free boundary $n$-harmonic map $\hat\Psi\colon (\Sigma^n,\mC)\to\mathbb{B}^{m}$ such that $(g,\rho)=(\alpha g_{\hat\Psi},\alpha^{-\frac{1}{2}}\rho_{\hat\Psi})$ for some $\alpha>0$ and $\bar\sigma_k(\Sigma^n,g_{\hat\Psi},\rho_{\hat\Psi}) = 1$. In particular, $\ind_S(\hat\Psi)\leqslant k$.

Conversely, let $\hat\Psi\colon (\Sigma^n,\mC)\to\mathbb{B}^{m}$ be a non-degenerate free boundary $n$-harmonic map. Then the pair $(g_{\hat\Psi},\rho_{\hat\Psi})$ is $\bar\sigma_{\ind_S(\hat\Psi)}$-conformally extremal.
\end{theorem}

\subsection{Remarks on the Steklov problem with density} The Steklov problem with density has not been previously mentioned in relation to sharp eigenvalue optimization problems. In this section we discuss how Theorem~\ref{Steklovconf:thm} fits with the existing results in the field and explain why it appears to be a natural setup for optimization of Steklov eigenvalues in the conformal class.   

\subsubsection{Upper bounds} Our first observation is that the normalized eigenvalues $\bar\sigma_k(\Sigma^n,g,\rho)$ are 
bounded independently of $(g,\rho)$, i.e. there exists a constant $C$ depending on the conformal class $[g]$ such that 
$$
\bar\sigma_k(\Sigma^n,g,\rho)\leqslant Ck^\frac{2}{n}
$$
For $\rho \equiv 1$ this is proved by Hassannezhad in~\cite[Theorem 4.1]{Has}. A slight modification of the proof yields the result for non-constant density $\rho$. It sufficient to repeat the proof of~\cite[Theorem 4.1]{Has} with $\bar\mu(A) := \int_{\partial\Sigma^n\cap A}\rho\,d\bar\mu_g$ in the notations of~\cite{Has}.

In particular, it makes sense to maximize the eigenvalues $\bar\sigma_k(\Sigma^n,g,\rho)$ in $[g]\times C^\infty_{>0}(\partial\Sigma^n)$ and investigate the existence and regularity of maximal pairs $(g,\rho)$. An analogous problem for Laplacian and Steklov eigenvalues on surfaces has been completely solved in the recent years~\cite{Pet1,Pet2,Pet3,KNPP2, NS}. In all these papers, the connection to harmonic maps is explicitly used in the proof. It seems natural that in order to have nice existence results in higher dimensions there has to be a connection of the problem to $n$-harmonic maps. This connection only manifests itself if one allows a non-trivial density $\rho$ to enter the picture. Thus, we believe that the Steklov problem with a density is a natural setup for optimization problems in a fixed conformal class.

\subsubsection{Fraser-Schoen's result for surfaces} 
\label{Steklovsurfaces:sec}
In the paper~\cite{FS3} Fraser and Schoen obtained the geometric characterizations of {\em maximal} metrics for Steklov and Laplacian eigenvalues on surfaces. However, their proofs can be adapted to the case of extremal metrics with only minor modifications. In fact, our analysis of extremal metrics in higher dimensions is heavily influenced by~\cite{FS3}. Here we compare Theorem~\ref{Steklovconf:thm} to the analogous result for surfaces~\cite[Proposition 2.8]{FS3} and explain why the densities do not appear for $n=2$.   

The normalized eigenvalues $\sigma_k(\Sigma^2,g,\rho)$ possess two properties specific to $n=2$. On one hand, since the Laplacian $\Delta_g$ is conformally covariant on surfaces, the harmonic extension is the same for all metrics in the conformal class $\mC$. Therefore, the eigenvalue $\sigma_k(\Sigma^2,g,\rho)$ depends on $g\in\mC$ only via the length of the normal vector. As a result, one has $\sigma_k(\Sigma^2,e^{2\omega}g,\rho)=\sigma_k(\Sigma^2,g,e^{-\omega}\rho)$. In particular, if $\hat\rho>0$ is any positive extension of $\rho$ to the interior, then $\sigma_k(\Sigma^2,g,\rho) = \sigma_k(\Sigma^2,\hat\rho^2g,1)$, i.e. the eigenvalues with density are a special case of classical Steklov eigenvalues.
On the other hand, the definition of the normalized eigenvalue does not include the volume of $\Sigma^2$. As a result, one has that $\bar\sigma_k(\Sigma^2,g,\rho) = \bar\sigma_k(\Sigma^2,\hat\rho^2g,1)$, i.e. the problem of optimizing normalized Steklov eigenvalues is the same whether one includes density or not.

Furthermore, the conformal invariance of the problem makes it impossible to identify extremal pairs as the ones induced by free boundary harmonic maps. Indeed, if $(g,\rho)$ is extremal, then $(e^{2\omega}g, e^{-\omega}\rho)$ is also extremal for all $\omega\in C^\infty(\Sigma^2)$. As a result, Theorem~\ref{Steklovconf:thm} takes the following form, which is a reformulation of~\cite[Proposition 2.8]{FS3} up to the conformal invariance described above.
 \begin{theorem}
\label{Steklovconfdim2:thm}
Let $\Sigma^2$ be a compact surface with boundary and $\mC$ be a conformal class on $\Sigma^2$. Suppose that the pair $(g,\rho)\in\mC\times C_{>0}^\infty(\partial\Sigma^n)$ is $\bar\sigma_k$-conformally extremal. Then there exists a non-degenerate free boundary harmonic map $\hat\Psi\colon (\Sigma^2,\mC)\to\mathbb{B}^{m}$ such that $(g,\rho)=(e^{2\omega} g_{\hat\Psi},e^{-\omega}\rho_{\hat\Psi})$ for some $\omega\in C^\infty(\Sigma^2)$ and $\bar\sigma_k(\Sigma^2,g_{\hat\Psi},\rho_{\hat\Psi}) = 1$. In particular, $\ind_S(\hat\Psi)\leqslant k$.

Conversely, let $\hat\Psi\colon (\Sigma^2,\mC)\to\mathbb{B}^{m}$ be a non-degenerate free boundary harmonic map. Then the pair $(e^{2\omega}g_{\hat\Psi},e^{-\omega}\rho_{\hat\Psi})$ is $\bar\sigma_{\ind_S(\hat\Psi)}$-conformally extremal for any $\omega\in C^\infty(\Sigma^2)$.
\end{theorem}

\subsubsection{Laplacian problem with density} Having seen that introducing density into a classical Steklov problem leads to a more geometrically natural optimization problem, one could ask whether the same happens for the Laplacian. Let $(M^n,g)$ be a closed Riemannian manifold and $\rho\in C^\infty(M^n)$. One defines the eigenvalues $\lambda_k(M^n,g,\rho)$ to be the the numbers such that the following equation has non-trivial solutions,
$$
\Delta_g u = \lambda\rho u.
$$    
In fact, recalling the definition of measure eigenvalues~\eqref{measureev:eq} one sees that $\lambda_k(M^n,g,\rho) = \lambda_k(M^n,g,\rho dv_g)$. In particular, by~\eqref{barlambdakmu:eq} the natural normalization is
$$
\bar\lambda_k(M^n,g,\rho) = \lambda_k(M^n,g,\rho)\vol(M^n,g)^\frac{2-n}{n}||\rho||_{L^1(M^n,g)}
$$
and one can study $\bar\lambda_k$-conformally extremal pairs $(g,\rho)$. However, it turns out that the density $\rho$ does not bring anything new to the problem. We prove in Theorem~\ref{Laplacianrho:thm} that for any conformally extremal pair $(g,\rho)$ one has that $\rho$ is a constant function. In particular $g$ is $\bar\lambda_k$-conformally maximal iff $(g,1)$ is $\bar\lambda_k$-conformally maximal. This fact gives further support to considering Steklov problem with a density in higher dimensions.  

\subsubsection{Other normalizations}  While the normalization~\eqref{barsigmak:eq} is natural from measure theory point of view, one could question whether other normalizations could lead to interesting optimization problems. Recall that up until now the Steklov eigenvalues in higher dimensions were considered with normalization either by $\vol(\Sigma^n,g)^{\frac{1}{n}}$ or $\vol(\partial\Sigma^n,g)^\frac{1}{n-1}$, see e.g~\cite{BFNT, CEG1, FS4, Has}. However, we show in Theorem~\ref{confSteklov:thm} that for any normalization different from~\eqref{barsigmak:eq} the corresponding optimization problem is not well-behaved. To be precise, we show the following.

\begin{proposition}
\label{nocrit:prop} 
For any $\alpha\ne 1$ the functionals
$$
F_{k,\alpha}(\Sigma^n,g) = \sigma_k(\Sigma^n,g)\vol(\partial \Sigma^n,g)^\alpha\vol(\Sigma^n,g)^{\frac{1+\alpha(1 - n)}{n}}
$$
do not have any smooth critical metrics $g$.
\end{proposition}

\subsection{Extremal metrics in the space of all metrics} Finally, we give the geometric characterization for critical points of eigenvalues functionals in the space of all metrics. In comparison to the fixed conformal class, the common feature for both problems is that the corresponding object is additionally required to be {\em conformal}.

We start with the Laplacian eigenvalues on a closed manifold $M^n$. Let $\Phi\colon (M^n,[g])\to\mathbb{S}^{m-1}$ be a non-degenerate $n$-harmonic map. The map $\Phi$ is called conformal if the pullback of the round metric on $\mathbb{S}^m$ is conformal to $g$, i.e. $\Phi^*g_{\mathbb{S}^{m-1}}\in [g]$. In particular, a direct computation shows $\Phi^*g_{\mathbb{S}^{m-1}} = g_\Phi := \frac{1}{n}|d\Phi|_g^2g$. Moreover, by~\cite[Corollary 4]{Takeuchi}  a non-degenerate conformal $n$-harmonic map is conformal iff its image is a minimal submanifold. Combining the last two observations, we observe that $\Phi\colon(M,g_\Phi)\to\mathbb{S}^{m-1}$ is an isometric minimal immersion.

We recall that the smooth metric $g$ on $M^n$ is called $\bar\lambda_k$-extremal if it is a critical point of the functional $g\mapsto \bar\lambda_k(M^n,g)$ in the space $\mR$ of all metrics on $M^n$. Evidently, any $\bar\lambda_k$-extremal metric is also $\bar\lambda_k$-conformally extremal, so there exists the corresponding $n$-harmonic map. The following theorem states that for $\bar\lambda_k$-extremal metric the $n$-harmonic map can be chosen to be conformal.

\begin{theorem}
\label{Laplace-extremal:thm}
Let $M^n$ be an $n$-dimensional closed manifold. Suppose that the smooth metric $g$ is $\bar\lambda_k$-extremal. Then there exists a minimal immersion $\Phi\colon M\to\mathbb{S}^m$ such that $g=\alpha \Phi^*g_{\mathbb{S}^{m-1}}$ for some $\alpha>0$ and $\bar\lambda_k(M^n,\Phi^*g_{\mathbb{S}^{m-1}}) = n$. In particular, $\ind_S(\Phi)\leqslant k$.

Conversely, let $\Phi\colon M\to\mathbb{S}^{m-1}$ be a minimal immersion. Then the metric $\Phi^*g_{\mathbb{S}^{m-1}}$ is $\bar\lambda_{\ind_S(\Phi)}$-extremal.
\end{theorem}
\begin{remark}
The same exact theorem is proved by El Soufi, Ilias in~\cite{ESI2}. However, our definition of extremal metric is slightly different, so we include the proof for completeness.
\end{remark}

Let $\Sigma^n$ be a compact manifold with non-empty boundary. A pair $(g,\rho)$ is called $\bar\sigma_k$-extremal if it is a critical point of $\bar\sigma_k$ in $\mR\times C^\infty_{>0}(\Sigma^n)$. Of course, such a pair is also $\bar\sigma_k$-conformally extremal and there exists a corresponding free boundary $n$-harmonic map $\hat\Psi\colon(\Sigma^n,g)\to\mathbb{B}^{m}$. The following theorem states that for $\bar\sigma_k$-extremal pairs the map $\hat\Psi$ can be chosen to be conformal. One consequence is that any extremal density $\rho$ is constant. Indeed, since $g_{\hat\Psi} = \Phi^*g_{\mathbb{B}^{m}}$ one has $\rho_{\hat\Psi} =  |\partial_{\nu_{g_{\hat\Psi}}} \hat\Psi| = |\nu_{g_{\hat\Psi}}|\equiv 1$. In particular, the introduction of density does not lead to non-trivial extremal densities.

\begin{theorem}
\label{Steklov-extremal:thm}
Let $\Sigma^n$ be an $n$-dimensional compact manifold with boundary, $n\geqslant 3$. Suppose that the pair $(g,\rho)\in\mR\times C_{>0}^\infty(\partial\Sigma^n)$ is $\bar\sigma_k$-extremal. Then there exists a free boundary minimal immersion $\hat\Psi\colon \Sigma^n\to\mathbb{B}^{m}$ such that $(g,\rho)=(\alpha \hat\Psi^*g_{\mathbb{B}^{m}},\alpha^{-\frac{1}{2}})$ for some $\alpha>0$ and $\bar\sigma_k(\hat\Psi^*g_{\mathbb{B}^{m}},1) = 1$. In particular, $\ind_S(\hat\Psi)\leqslant k$.

Conversely, let $\hat\Psi\colon\Sigma^n\to\mathbb{B}^{m}$ be a free boundary minimal immersion. Then the pair $(\hat\Psi^*g_{\mathbb{B}^{m}},1)$ is $\bar\sigma_{\ind_S(\hat\Psi)}$-extremal.
\end{theorem}

\subsection{Free boundary harmonic annuli} 
\label{fbha:sec}
Let $\mC_T$ be the conformal class on the annulus containing the flat metric on $\mathbb{A}_T = [0,t]\times \mathbb{S}^1$. By the uniformization theorem any metric on an annulus is in $\mC_T$ for some $T > 0$. In Section~\ref{Annuli:sec} we construct explicit examples of rotationally symmetric free boundary harmonic maps $\hat\Psi_T\colon \mathbb{A}_T \to \mathbb{B}^3$ of spectral index $1$. Geometrically, their images are pieces of stretched catenoids, see Figure~\ref{fig:examples}. In particular, they correspond to $\bar\sigma_1$-conformally extremal pairs. Our main motivation is to demonstrate that there are free boundary $n$-harmonic maps $\hat\Psi$ such that the corresponding density $\rho_{\hat\Psi}$ is not constant, i.e. the introduction of densities is indeed necessary for the geometric characterization. To that end we show that for a particular range of $T$ the densities $\rho_{\hat\Psi_T}$ are locally constant, but not identically constant, i.e. they take different values on different boundary components.

Our analysis is reminiscent of that in~\cite{FS1}, where the authors study the first Steklov eigenvalue of rotationally symmetric metrics on $\mathbb{A}_T$ --- see also \cite{FSar, FTY} for higher eigenvalues. Moreover, Fraser and Schoen proved in~\cite{FS2} that the only $\bar\sigma_1$-extremal metric on the annulus is the metric on the critical catenoid. The corresponding value of $T$ is $T_1=2t_1$, where $t_1\approx 1.2$ is the unique solution to $t=\coth t$.

\begin{theorem}
\label{annulus:thm}
For any $T\geqslant T_1$ there exists a rotationally symmetric free boundary harmonic map $\hat\Psi_T\colon \mathbb{A}_T\to\mathbb{B}^3$ of spectral index $1$. Furthermore, the corresponding $\bar\sigma_1$-conformally extremal pair $(g_{\hat\Psi_T},\rho_{\hat\Psi_T})$ possesses the following properties.
\begin{enumerate}
\item $\hat\Psi_{T_1}$ is the immersion of the minimal catenoid.
\item The pair $(g_{\hat\Psi_T},\rho_{\hat\Psi_T})$ is rotationally symmetric, in particular, $\rho_{\hat\Psi_T}$ is constant on each boundary component. There exists $T_2$ such that for $T\geqslant T_2$ one has 
$\rho_{\hat\Psi_T}(0)\ne\rho_{\hat\Psi_T}(T)$.
\item One has
\begin{equation}
\label{AT:eq1}
\bar\sigma_1(\mathbb{A}_T,g_{\hat\Psi_T},\rho_{\hat\Psi_T}) = 2E(\hat\Psi_T)>2\pi
\end{equation}
and 
\begin{equation}
\label{AT:eq2}
\lim_{T\to\infty} \bar\sigma_1(\mathbb{A}_T,g_{\hat\Psi_T},\rho_{\hat\Psi_T}) = 2\pi.
\end{equation}
\end{enumerate}
\end{theorem}

In fact, we conjecture that the $(g_{\hat\Psi_T},\rho_{\hat\Psi_T})$ is a $\bar\sigma_1$-conformally maximal pair. Keeping in mind the conformal invariance of Section~\ref{Steklovsurfaces:sec} this can be formulated in the following way.
\begin{conjecture}
    For any $T \geqslant T_1$, let $\mC_T$ be the conformal class containing the flat metric on $[0,T] \times \mathbb{S}^1$. Then
$$
\sup_{g\in \mC_T}\bar\sigma_1(\mathbb{A},g) = 2E(\hat\Psi_T),
$$ 
where $\hat\Psi_T$ is the free boundary harmonic map obtained in Theorem~\ref{annulus:thm}.

In particular, for these conformal classes $\bar\sigma_1$-conformally maximal pairs can be chosen to be rotationally symmetric.
\end{conjecture} 

Properties (1), (3) of Theorem~\ref{annulus:thm} are consistent with this conjecture. Indeed, it is proven in~\cite{MP} that for any conformal class $\mC$ on the annulus $\mathbb{A}$ one has $\sup_{g\in\mC}\bar\sigma_1(\mathbb{A},g)>2\pi$, which agrees with~\eqref{AT:eq1}. At the same time, Medvedev in~\cite{Med} studied the behaviour of maximizers under the conformal degeneration. In particular, he established that the relation~\eqref{AT:eq2} holds for $\bar\sigma_1$-conformally maximal metrics.

Surprisingly enough, the considerations used to prove Theorem~\ref{annulus:thm} imply that for $T<T_1$ there are no rotationally symmetric free boundary harmonic maps  $\hat\Psi_T\colon \mathbb{A}_T\to\mathbb{B}^3$ of spectral index $1$. Since by~\cite{FS2, KKP} the multiplicity of $\sigma_1$ on any annulus can not exceed $3$, either for $T<T_1$ the $\bar\sigma_1$-conformally maximal metric is not rotationally symmetric or the corresponding map has the image in $\mathbb{B}^2$. The latter situation seems unlikely and we expect that for $T<T_1$ the $\bar\sigma_1$-conformally maximal pairs in $\mC_T$ are no longer rotationally symmetric. We can not prove this in full generality, but the following  holds.

\begin{proposition} \label{annulus:prop}
There exists $\tilde{T}<T_1$ such that for all $T\leqslant \tilde{T}$ the conformal class $\mC_T$ does not have rotationally symmetric $\bar\sigma_1$-conformally maximal pairs.
\end{proposition}

In fact, this behaviour for small $T$ is supported by the analysis of conformal degenerations in~\cite{Med} and the example in~\cite{MP}. Both papers suggest that as $T\to 0$ the limit of corresponding free boundary harmonic maps is a single boundary bubble, which is impossible for rotationally symmetric maps.

\section{Algebraic extremality conditions}
\label{alg_extremal:sec}

In this section we obtain algebraic conditions condition on extremal metrics. They are used later in Section~\ref{geometricextremal:sec} to complete the geometric characterization of extremal metrics. 

\begin{definition}
\label{extremal:def}
    We say that a metric $g$ is $F$-extremal for some functional $F$ if for all one-parameter smooth family of metrics $g(t)$ with
    $g(0) = g$, we have either
    $$F(g(t)) \leqslant F(g) + o(t) \qquad \text{or} \qquad F(g(t)) \geqslant F(g) + o(t)$$
    as $t \to 0$.

    A metric is $F$-conformally extremal if it is $F$-extremal in the conformal class of $g$. 
\end{definition}
In the following the functional $F$ will be a (Laplace or Steklov) eigenvalue with the appropriate normalization. This definition was introduced by Nadirashvili in~\cite{NadirashviliTorus} in the context of normalized Laplace eigenvalues.

\subsection{Extremality conditions for a fixed conformal class}

\begin{theorem}
\label{Laplacianconfalg:thm}
Let $M^n$ be an $n$-dimensional closed manifold, $n\geqslant 3$ and let $\mC$ be a conformal class of metrics on $M^n$. Suppose that the metric $g\in \mC$ is conformally extremal for the functional 
$$
\bar\lambda_k(M^n,g) = \lambda_k(M^n,g) \vol(M^n,g)^\frac{2}{n}.
$$
Then there exists a collection $u_1,\ldots, u_m$ of $\lambda_k(M^n,g)$-eigenfunctions such that 
\begin{itemize}
    \item[1)] $\sum_{i=1}^m u_i^2 = \frac{1}{\lambda_k}$
    \item[2)] $\sum_{i=1}^m |d u_i|^2_g = 1$
\end{itemize}
Conversely, if there exists a collection of $\lambda_k(M^n,g)$-eigenfunctions satisfying $1)-2)$ and additionally $\lambda_k(M^n,g)>\lambda_{k-1}(M^n,g)$ or $\lambda_k(M^n,g)<\lambda_{k+1}(M^n,g)$, then $g$ is extremal for the functional $\bar\lambda_k(M^n,g)$ in $\mC$.
\end{theorem}

\begin{proof}
    The proof follows Fraser-Schoen \cite{FS3} and El Soufi-Ilias \cite{ESI2} arguments. 

    For any smooth family $g(t) = e^{f(t)} g$ with $f(0) = 0$, we write $\lambda_k(M,g(t)) = \lambda_k(t)$.
    It is known that $\lambda_k(t)$ is Lipschitz. Furthermore for almost all $t$, when $\dot\lambda_k(t)$
    exists we have
    \begin{align*}
        \dot\lambda_k(t) = \int_M \left(\frac{n-2}{2} |\nabla u|^2 - \frac{n \lambda_k}{2} u^2\right) \dot{f}
    dv_g =: Q_{\dot{f}(t)}(u)
    \end{align*}
    where $u \in E_k(g(t))$. Indeed, for almost all $t_0$, we have a neighborhood around $t_0$ on which the 
    multiplicity of $\lambda_k(t)$ is constant and there exists $l$ with $\lambda_l(t) = \lambda_k(t)$,
    $\lambda_{l-1}(t) < \lambda_l(t)$. Fix such a $t_0$ and define on this neighborhood
    $\mathcal{E}(t) = \bigcup_{j = 0}^{l-1} E_j(g(t))$ and the orthogonal projection on $\mathcal{E}(t)$,
    $P_t : L^2(M) \to \mathcal{E}(t)$. Take $u_0 \in E_k(g(t_0))$ with $\|u_0\|_{L^2(M,g(t_0))} = 1$
    and let $u_t = u_0 - P_t(u_0)$. Then the functional 
    \begin{align*}
        F(t) = \int_M |\nabla u_t|^2 dv_{g(t)} - \lambda_l(t) \int_M u_t^2 dv_{g(t)}
    \end{align*}
    satisfies $F(t) \geqslant 0$ and the neighborhood of $t_0$ and $F(t_0) = 0$ so $\dot{F}(t_0) = 0$.
    Computing $\dot{F}(t_0)$ gives the desired formula for $\dot{\lambda}_k(t_0)$. 

    To continue the proof we need a lemma stating that for any perturbation of an extremal metrics we can find
    eigenfunctions for which $Q$ is null, similar to lemma 2.3 in Fraser-Schoen \cite{FS3}, but its proof
    is adapted to our definition of extremal metric. 
    \begin{lemma} \label{lem:null-direction}
        For any $\phi \in L^2(M)$ with $\int_M \phi dv_g = 0$, there exists $u \in E_k(g)$ with
        $\|u\| = 1$ such that $Q_\phi(u) = 0$. 
    \end{lemma}
    \begin{proof}
        We approximate $\phi \in L^2(M)$ by a sequence of $\phi_j \in C^\infty$ such that
        $\phi_j \to \phi$ and $\int_M \phi_j dv_g = 0$ for all $j$. Let $g_j(t) = \frac{(1 + t\phi_j)g}{\vol(M,(1+t\phi_j)g)}$. Then $\frac{d}{dt}g_j(0) = \phi_j g$ and $\vol(M,g_j(t)) = 1$. 

        Since $g$ is extremal for $\lambda_k$ and $\vol(M,g_j(t)) = 1$ we can assume without loss of generality that
        \begin{align*}
            \lambda_k(g_j(t)) \leqslant \lambda_k(g) + o(t)
        \end{align*}
        as $t \to 0$. In particular, taking the limit from the left gives
        \begin{align*}
            \lim_{t \to 0^-} \frac{\lambda_k(g_j(t)) - \lambda_k(g)}{t} \geqslant 0
        \end{align*}
        so there exists a sequence of $\epsilon_i > 0$ decreasing to 0 and $\delta_i$ with $\lim \delta_i \geqslant 0$
        such that
        \begin{align*}
            \delta_i &\leqslant \frac{\lambda_k(g_j(-\epsilon_i)) - \lambda_k(g)}{\epsilon_i} 
            = \frac{1}{\epsilon_i} \int_{-\epsilon_i}^0 \dot\lambda_k(g_j(t)) dt \\
                     &\leqslant \esssup_{t \in [-\epsilon_i, 0]} \dot\lambda_k(g_j(t)).
        \end{align*}
        Hence we can find a sequence of $t_i < 0$ increasing to $0$ such that
        $\dot\lambda_k(g_j(t_i)$ exists and $\exists u_i^{(j)} \in E_k(g_j(t_i)), \|u_i^{(j)}\|_{L^2(g_j(t_i))} = 1$
        with $Q_{\phi_j} (u_i^{(j)}) \geqslant \delta_i$.
        Then, after taking a subsequence if necessary, $u_i^{(j)} \to u_{-}^{(j)}$ in $C^2(g)$
        with $u_{-}^{(j)} \in E_k(g), \|u_{-}^{(j)}\|_{L^2(g)} = 1$ and $Q_{\phi_j} (u_{-}^{(j)}) \geqslant 0$.
        Again taking a subsequence if necessary, we have $u_{-}^{(j)} \to u_{-}$ in $C^2(g)$ with
        $u_{-} \in E_k(g), \|u_{-}\|_{L^2(g)} = 1$ and $Q_\phi (u_{-}) \geqslant 0$. 

        The same process starting with the limit from the right gives $u_{+} \in E_k(g)$
        with $\|u_{+}\|_{L^2(g)} = 1$ and $Q_\phi(u_{+}) \leqslant 0$. Then taking a linear combination
        of $u_{+}$ and $u_{-}$ gives the desired $u$. 
    \end{proof}

    Let $K$ be the convex hull in $L^2(M,g)$ of $\left\{ \frac{2 - n}{2} |\nabla u|^2 + \frac{n\lambda_k(g)}{2} u^2
    \,\big|\, u \in E_k(g)\right\}\!$. Then $1 \in K$. If not, then by Hahn-Banach's theorem, there exists
    $\phi \in L^2(M,g)$ such that
    \begin{align*}
        0 &< \int_M \left(\frac{2 - n}{2}|\nabla u|^2 + \frac{n \lambda_k(g)}{2} u^2\right)\phi dv_g= - Q_\phi(u)
        \quad \forall u \in E_k(g) \\
        0 &> \int_M \phi dv_g.
    \end{align*}
   Taking $\tilde{\phi} = \phi - \frac{1}{\vol(M,g)} \int_M \phi$ we obtain a perturbation satisfying the
   conditions of the previous lemma so there exists $u \in E_k(g)$ such that
   \begin{align*}
       0 &= Q_{\tilde{\phi}}(u) = Q_\phi(u) - \frac{\int_M \phi dv_g}{\vol(M,g)} \int_M \frac{n-2}{2}|\nabla u|^2
       - \frac{n \lambda_k}{2} u^2 dv_g \\
         &= Q_{\phi}(u) + \frac{\int_M \phi dv_g}{\vol(M,g)} \int_M |\nabla u|^2 dv_g \\
         &<0
   \end{align*}
   a contradiction. Hence $1 \in K$ and there exists $u_1, \dots, u_m \in E_k(g)$ such that
   \begin{align} \label{eq:laplace-conf-res}
       1 = \sum_{j = 1}^m \frac{2-n}{2} |\nabla u_j|^2 + \frac{n \lambda_k(g)}{2} u_j^2.
   \end{align}
   This implies that $\sum_{j = 1}^m \lambda_k(g) u_j^2 = 1$ on $M$ by considering 
   the function $F = \sum_{j = 1}^m u_j^2 - \frac{1}{\lambda_k}$
   and remarking that $(n-2)\Delta F = -4\lambda_k F$ hence $F = 0$. 
   The second equality of the theorem then follows directly.

   We now prove the converse. Suppose 
   $\lambda_k > \lambda_{k-1}$ (the case $\lambda_k < \lambda_{k+1}$ is similar) and there exists eigenfunctions
   $u_1, \dots, u_m \in E_k(g)$ satisfying the conditions $1)-2)$. Let $g(t) = e^{f(t)}g$ be a smooth family of
   metrics on $M$.  We can assume without loss of generality that $g(t)$ keeps the normalisation 
   $\vol(M,g)^{\frac{2}{n}}$ constant, thus  $\int_M f dv_g = 0$. Write $E = \Span \{u_1, \dots, u_m\}$ then
   \begin{align*}
       \sum_{j = 1}^m  Q_{\dot{f}}(u_j)
           &= \int_M \left(\frac{n-2}{2} \sum_{j = 1}^m |\nabla u_j|^2 
           - \frac{n \lambda_k}{2} \sum_{j = 1}^m u_j^2\right) \dot{f} dv_g \\
           &= \int_M \left(\frac{n-2}{2} - \frac{n}{2}\right) \dot{f} dv_g = 0,
   \end{align*}
   so there exists $u_\pm \in E$ with $\pm Q_{\dot{f}} (u_\pm) \leqslant 0$. 

   Since $\lambda_k(g) > \lambda_{k-1}(g)$, $Q_{\dot{f}} (u_+)\leqslant 0$ implies that
   $\lim_{t \to 0^+} \frac{\lambda_k(g(t)) - \lambda_k(g)}{t} \leqslant 0$, while
   $Q_{\dot{f}}(u_-) \leqslant 0$ gives $\lim_{t \to 0^-} \frac{\lambda_k(g(t)) - \lambda_k(g)}{t} \geqslant 0$
   hence $\lambda_k(g(t)) \leqslant \lambda_k(g) + o(t)$. The family $g(t)$ was arbitrary so we
   conclude that $g$ is extremal. 
\end{proof}

\begin{theorem}
\label{Laplacianrho:thm}
Let $M^n$ be an $n$-dimensional closed manifold, $n\geqslant 3$ and let $\mC$ be a conformal class of metrics on $M^n$. Suppose that the pair $(g,\rho)\in \mC\times C^\infty(M)$ is extremal for the functional
$$
\bar\lambda_k(M^n,g,\rho) = \lambda_k(M^n,g,\rho)\vol(M^n,g)^\frac{2-n}{n}\left(\int_{M^n}\rho\,dv_g\right)
$$
in $\mC\times C^\infty(M^n)$. Then $\rho$ is a constant function $\rho\equiv \rho_0$ and 
there exists a collection $u_1,\ldots, u_m$ of $\lambda_k(M^n,g,\rho)$-eigenfunctions such that 
\begin{itemize}
    \item[1)] $\sum_{i=1}^m u_i^2 = \frac{1}{\lambda_k \int_M \rho dv_g}$
    \item[2)] $\sum_{i=1}^m |d u_i|^2_g = \frac{1}{\vol(M,g)}$
\end{itemize}
Conversely, if $\rho\equiv\rho_0$ is a constant function and 
there exists collection $\lambda_k(M^n,g,\rho)$-eigenfunctions satisfying $1)-2)$ and additionally $\lambda_k(M^n,g,\rho)>\lambda_{k-1}(M^n,g,\rho)$ or $\lambda_k(M^n,g,\rho)<\lambda_{k+1}(M^n,g,\rho)$, then $(g,\rho)$ is extremal for the functional $\bar\lambda_k(M^n,g,\rho)$ in $\mC\times C^\infty(M^n)$.
\end{theorem}
\begin{proof}
    The argument is the same than for the previous theorem, with slight modification to account
    for the density $\rho$. For $g(t) = e^{f(t)}g$, the derivative of $\lambda_k(g(t))$ (when
    it exists) is now given by
    \begin{align*}
        \dot{\lambda}_k(t) = \int_M \left(\frac{n-2}{2} |\nabla u|^2 - \frac{n}{2} \lambda_k(t) u^2 \rho\right) \dot{f}
        - \lambda_k(t) u^2 \dot{\rho} \,dv_{g(t)} =: Q_{(\dot{f}, \dot{\rho})} (u)
    \end{align*}
    where $u \in E_k(g(t))$ with $\|u\|_{L^2(M,\rho dv_{g(t)})} = \int_M u^2 \rho dv_{g(t)} = 1$. 

    We have a result similar to lemma \ref{lem:null-direction}: if $(\phi, \eta) \in L^2(M) \times
    L^2(M)$ with
    \begin{align*}
        \frac{2 - n}{2} \frac{1}{\vol(M, g)} \int_M \phi dv_g + \frac{1}{\int_M \rho dv_g} \int_M \eta + \frac{n}{2} \phi \rho dv_g = 0
    \end{align*}
    then there exists $u \in E_k(g)$ with $\|u\|_{L^2(M, \rho dv_g)} = 1$ such that $Q_{(\phi, \eta)}(u) = 1$.

    Letting $\mathcal{H} = L^2(M) \times L^2(M)$ with the inner product 
    $$\ip{(f_1, h_1),(f_2,h_2)} = \int_M f_1 f_2 dv_g  + \int_M h_1 h_2 dv_g$$
    and $K$ the convex hull of
    $$\left\{ \left( \frac{2 - n}{2} |\nabla u|^2 + \frac{n}{2} \lambda_k u^2 \rho,
    \lambda_k u^2\right) , u \in E_k(g,\rho) \right\}$$
    in $\mathcal{H}$, we have 
    $\left( \frac{2 - n}{2} \frac{1}{\vol(M,g)} + \frac{n}{2} \frac{1}{\int_M \rho dv_g} \rho,
    \frac{1}{\int_M \rho dv_g}\right) \in K$. If not we could use Hahn-Banach's theorem to obtain
    a contradiction. Hence there exists eigenfunctions $u_1, \dots, u_m$ such that
    \begin{align*}
        \begin{cases}
            \sum_{j = 1}^m \left( \frac{2 - n}{2} |\nabla u_j|^2 + \frac{n}{2} \lambda_k u_j^2 \rho\right)
            = \frac{2 - n}{2} \frac{1}{\vol(M,g)} + \frac{n}{2} \frac{1}{\int_M \rho dv_g} \rho \\
            \sum_{j = 1}^m \lambda_k u_j^2 = \frac{1}{\int_M \rho dv_g}.
        \end{cases}
    \end{align*}
    Using the second equation in the first one gives $\sum |\nabla u_j|^2 = \frac{1}{\vol(M,g)}$. 
    Finally, since $\sum u_j^2$ is constant on $M$, 
    \begin{align*}
        0 &= \Delta \Big(\sum_{j=1}^m u_j^2\Big) = 2 \sum_{j = 1}^2 u_j \Delta u_j + 2 \sum_{j = 1}^m |\nabla u_j|^2 \\
          &= -2 \sum_{j = 1}^m \rho u_j^2 + \frac{2}{\vol(M,g)} = - \frac{2 \rho}{\lambda_k(g) \vol(M,\rho)} + \frac{2}{\vol(M,g)}
    \end{align*}
    so $\rho$ is constant with $\rho = \frac{\lambda_k(g) \vol(M,\rho)}{\vol(M,g)}$ on $M$.

    The proof of the converse is the same than for the no-density case and is thus omitted.
\end{proof}
\begin{remark}
The theorem above essentially states that introducing density does not add any new critical points. In particular, it suggests that the optimal isoperimetric inequality should be the same regardless of the presence of density function. In fact when considering conformally maximal metrics we have the corollary
\end{remark}
\begin{corollary}
    The metric $g$ is $\bar\lambda_k$-conformally maximal if and only if $(g,1)$ is $\bar\lambda_k$-conformally maximal.
\end{corollary}
\begin{proof}
    That $(g,1)$ being $\bar\lambda_k$-conformally maximal implies that $g$ is $\bar\lambda_k$-conformally maximal is clear. For the
    other direction, if $g$ is $\bar\lambda_k$-conformally maximal then there exists $u_1, \dots, u_m$ $\lambda_k$-eigenfunctions
    satisfying:
    \begin{align*}
        \sum_{j = 1}^m u_j^2 = \frac{1}{\lambda_k(g)} \\
        \sum_{j = 1}^m |du_j|^2 = 1.
    \end{align*}
    Then after rescaling the $u_j$'s,
    \begin{align*}
        \sum_{j = 1}^m u_j^2 = \frac{1}{\lambda_k(g) \int_M \rho dv_g} \\
        \sum_{j = 1}^m |du_j|^2 = \frac{1}{\vol(M,g)}
    \end{align*}
    where $\rho = 1$. 

    Finally since $g$ is $\bar\lambda_k$-conformally maximal and the conformal spectrum is simple \cite{CES}, we have $\lambda_k(g) > \lambda_{k-1}(g)$.
    So by the converse of theorem \ref{Laplacianrho:thm}, $(g,1)$ is also $\bar\lambda_k$-conformally maximal.
\end{proof}

\begin{theorem}
\label{confSteklov:thm}
Let $\Sigma^n$ be an $n$-dimensional connected compact manifold with non-empty boundary, $n\geqslant 3$ and let $\mC$ be a conformal class of metrics on $\Sigma^n$. Let $\alpha \in \mathbb{R}$ and suppose that the metric $g\in \mC$ is extremal for the functional 
$$
F_{k,\alpha}(\Sigma^n,g) = \sigma_k(\Sigma^n,g)\vol(\partial \Sigma^n,g)^\alpha\vol(\Sigma^n,g)^{\frac{1+\alpha(1 - n)}{n}}
$$
in $\mC$.
Then $\alpha =1$, i.e. $F_{k,\alpha} = \bar\sigma_k$, and  there exists a collection $u_1,\ldots, u_m$ of $\sigma_k(\Sigma^n,g)$-eigenfunctions such that 
\begin{itemize}
    \item[1)] $\sum_{i=1}^m u_i^2 = \frac{1}{\sigma_k(g) \vol(\del \Sigma, g)}$ on $\partial \Sigma^n$;
    \item[2)] $\sum_{i=1}^m |d u_i|^2_g = \frac{1}{\vol(\Sigma, g)}$ on $\Sigma^n$.
\end{itemize}
Conversely, if there exists a collection of $\sigma_k(\Sigma^n,g)$-eigenfunctions satisfying $1)-2)$ 
and additionally $\sigma_k(\Sigma^n,g)>\sigma_{k-1}(\Sigma^n,g)$ or $\sigma_k(\Sigma^n,g)<\sigma_{k+1}(\Sigma^n,g)$, then $g$ is extremal for the functional $\bar\sigma_k(\Sigma^n,g) = F_{k,1}(\Sigma^n,g)$ in $\mC$. 
\end{theorem}
\begin{remark}
This theorem states that $\bar\sigma_k(\Sigma^n,g)$ is the only normalization of $\sigma_k(\Sigma^n,g)$ by powers of $\vol(\Sigma^n,g)$ and $\vol(\partial \Sigma^n,g)$, where one could expect the existence of regular maximizers in the conformal class.  
\end{remark}

\begin{proof}
    We use once again the same arguments.
    For $(\phi, \psi) \in L^2(\Sigma) \times L^2(\del \Sigma)$ and $u \in C^\infty (\Sigma)$ let 
    \begin{align*}
        Q_{(\phi,\psi)}(u) = -\int_\Sigma \left(1 - \frac{n}{2}\right) |\nabla u|^2 \phi - \int_{\del \Sigma} \frac{n-1}{2} \sigma_k(t) u^2 \psi.
    \end{align*}
    Then, if $g(t) = e^{f(t)}g$, the derivative of $\sigma_k(g(t))$ is given almost everywhere (when it exists) by
    \begin{align*}
        \dot\sigma_k(t) = Q_{(\dot{f},\dot{f})}(u)
    \end{align*}
    where $u \in E_k(g(t))$ with $\|u\|_{L^2(\del \Sigma)} = 1$. 

    Lemma \ref{lem:null-direction} can be adapted to the Steklov problem to obtain
    \begin{lemma} \label{lem:null-direction-stek}
        For any $(\phi, \psi) \in L^2(\Sigma) \times L^2(\del \Sigma)$ with
        \begin{align*}
            \frac{(n-1) \alpha}{\vol(\del \Sigma, g)} \int_{\del \Sigma} \psi + \frac{1 + \alpha(1-n)}{\vol(\Sigma,g)} \int_{\Sigma} \phi = 0
        \end{align*}
        there exists $u \in E_k(g)$, $\|u\|_{L^2(\del \Sigma)} = 1$ such that $Q_{(\phi,\psi)}(u) = 0$. 
    \end{lemma}
    \begin{proof}
        We can approximate $(\phi,\psi)$ by a sequence of $f_i \in C^\infty(M)$ such that
        \begin{align} 
            \frac{(n-1) \alpha}{\vol(\del \Sigma, g)} \int_{\del \Sigma} f_i + \frac{1 + \alpha(1-n)}{\vol(\Sigma,g)} \int_{\Sigma} f_i = 0
        \end{align}
        and $(f_i, f_i) \to (\phi,\psi)$ in $L^2$. Let $g_i(t) = r(t) (1 + tf_i) g$ where $r(t)$ is such that
        $\vol(\del \Sigma, g_i(t))^\alpha \vol(\Sigma, g_i(t))^{\frac{1 + \alpha(1 - n)}{n}} = 1$. Then en by construction of the 
        $f_i$, we have $\frac{d}{dt} g_i(t) = f_i g$.

        Without loss of generality, since $g$ is $F_{k,\alpha}$-conformally extremal and the volumes normalisation
        is kept constant, we can assume that $\sigma_k(g(t)) \leqslant \sigma_k(g) + o(t)$ as $t \to 0$. As in lemma \ref{lem:null-direction}, the limit from the
        left gives $u_{-} \in E_k(g)$ with $\|u_{-}\|_{L^2(\del \Sigma)} = 1$ and $Q_{(\phi,\psi)}(u_{-}) \geqslant 0$, while the limit from the right gives
        $u_{+} \in E_k(g)$ with $\|u_{+}\|_{L^2(\del \Sigma)} = 1$ and $Q_{(\phi, \psi)}(u_{+}) \leqslant 0$. Taking a linear combination of $u_{-}$ and $u_{+}$
        gives the desired $u$. 
    \end{proof}

    Let $K$ be the convex hull in $\mathcal{H} = L^2(\Sigma) \times L^2(\del \Sigma)$ of pairs of functions $\{(\frac{2 - n}{2}|\nabla u|^2,\frac{(n-1)\sigma_k}{2} u^2) \,\big|\, u \in E_k(g)\}$. We claim
    that 
    $$
    \left(\sgn_+(\alpha) \frac{1  + \alpha(1 -n)}{\vol(\Sigma)}, \sgn_+(\alpha) \frac{(n-1) \alpha}{\vol(\del \Sigma)}\right)\in K,
    $$
      where $\sgn_+(\alpha) = +1$ if $\alpha \geqslant 0$, $=-1$ otherwise. If not by Hahn-Banach theorem,
    there exists $(\phi,\psi)$ such that
    \begin{align*}
        \sgn_+(\alpha) \Big\langle (\phi,\psi), \left(\frac{1 + \alpha(1 - n)}{\vol(\Sigma)}, \frac{(n-1)\alpha}{\vol(\del \Sigma)}\right)\Big\rangle_{\mathcal{H}}  &> 0 \\
        \Big\langle(\phi,\psi), \left((1-\frac{n}{2}) |\nabla u|^2, \frac{(n-1)\sigma_k}{2} u^2\right)\Big\rangle_{\mathcal{H}} &< 0 \quad \forall u \in E_k(g) \setminus \{0\}.
    \end{align*}
    Setting  for $\alpha \neq 0$,
    \begin{align*}
        \tilde{\psi} = \psi - \frac{1}{\vol(\del \Sigma)} \int_{\del \Sigma} \psi - \frac{1 + \alpha(n-1)}{\alpha(n-1) \vol(\Sigma)} \int_{\Sigma} \phi
    \end{align*}
    we have that $(\phi, \tilde{\psi})$ satisfy the condition of lemma \ref{lem:null-direction-stek} and there exists $u \in E_k(g)$ such that
    \begin{align*}
        & 0 = Q_{(\phi, \tilde{\psi})}(u) = Q_{(\phi,\psi)}(u) -  \\
           & - \frac{(n-1)\sigma_k}{2|\alpha|(n-1)} \sgn_+(\alpha)
        \left\langle(\phi,\psi), \left(\frac{1 + \alpha(n-1)}{\vol(\Sigma)}, \frac{\alpha (n-1)}{\vol(\del \Sigma)}\right)\right\rangle
        \int_{\del \Sigma} u^2 
          < 0.
    \end{align*}
    The same contradiction is obtained for $\alpha = 0$ by instead considering $(\tilde{\phi}, \psi)$
    with
    \begin{align*}
        \tilde{\phi} = \phi - \frac{1}{\vol(\Sigma,g)} \int_\Sigma \phi.
    \end{align*}
    This proves the claim and we conclude that
    there exists $u_1, \dots, u_m \in E_k(g)$ such that
    \begin{align*}
        \begin{cases}
            \sum_{j = 1}^m \left(1 - \frac{n}{2}\right) |\nabla u_j|^2 = \sgn_+(\alpha) \frac{1 + \alpha (1 - n)}{\vol(\Sigma)} \\
            \sum_{j = 1}^m \frac{\sigma_k(g)}{2} u_j^2 = \sgn_+(\alpha) \frac{\alpha}{\vol(\del \Sigma)}
        \end{cases}
    \end{align*}
    Integrating the first equation on $\Sigma$ and using that the $u_j$ are eigenfunctions gives
    $1 + \alpha(1 - n) = \alpha (2-n)$.
    Solving this equation for $\alpha$ yields the necessary condition $\alpha = 1$. 
    The previous system becomes
    \begin{align*}
        \begin{cases}
            \sum_{j = 1}^m |\nabla u_j|^2 = \frac{2}{\vol(\Sigma)} \\
            \sum_{j = 1}^m u_j^2 = \frac{2}{\sigma_k(g) \vol(\del \Sigma)}
        \end{cases}
    \end{align*}
    which after rescaling the $u_j$'s are the desired result.

    For the converse
    we will only treat the case $\sigma_k(\Sigma,g) > \sigma_{k-1}(\Sigma,g)$, the other
    case being similar. Suppose that we have the $\sigma_k(g)$-eigenfunctions
    $u_1, \dots, u_m$ satisfying $1)-2)$. Let $g(t) = e^{f(t)}$ smooth family of metrics
    which keeps the normalisation $\vol(\del \Sigma, g) \vol(\Sigma, g)^{\frac{2 -n}{n}}$ constant,
    hence 
    \begin{align*}
        \frac{n-1}{\vol(\del \Sigma,g)} \int_{\del \Sigma} \dot{f} + \frac{2 - n}{\vol(\del \Sigma,g)} \int_\Sigma \dot{f} = 0.
    \end{align*}
    Let $E = \Span\{u_1, \dots, u_m\}$ then
    \begin{align*}
        \sum_{j = 1}^m Q_{(\dot{f}, \dot{f})}(u_j) 
        &= - \int_\Sigma \left(1 - \frac{n}{2}\right) \sum_{j = 1}^m |\nabla u_j|^2 \dot{f}
        - \int_{\del \Sigma} \frac{n-1}{2} \sum_{j = 1}^m \sigma_k u_j^2 \dot{f} \\
        &= - \frac{2-n}{2 \vol(\Sigma,g)} \int_\Sigma \dot{f} - \frac{n-1}{2 \vol(\del \Sigma, g)}
        \int_{\del \Sigma} \dot{f} \\
        &= 0.
    \end{align*}
    Then there exists $u_\pm \in E$ such that $\pm Q_{(\dot{f},\dot{f})}(u_\pm) \leqslant 0$. 
    The fact that $\lambda_k(g) > \lambda_{k-1}(g)$ allows us to get 
    $\lim_{t \to 0^+} \frac{\lambda_k(g(t)) - \lambda_k(g)}{t} \leqslant 0$ from $Q_{(\dot{f}, \dot{f})}(u_+)$,
    and we have $\lim_{t \to 0^-} \frac{\lambda_k(g(t)) - \lambda_k(g)}{t} \geqslant 0$ from $Q_{(\dot{f}, \dot{f})}(u_-)$.
    Thus $\lambda_k(g(t)) \leqslant \lambda_k(g) + o(t)$ and $g$ is extremal.
\end{proof}

\begin{theorem}
\label{Steklovrho:thm}
Let $\Sigma^n$ be an $n$-dimensional connected compact manifold with non-empty boundary, $n\geqslant 3$ and let $\mC$ be a conformal class of metrics on $M$. Suppose that the pair $(g,\rho)\in \mC\times C^\infty(\partial \Sigma^n)$ is extremal for the functional 
$$
\bar\sigma_k(\Sigma^n,g,\rho) = \sigma_k(\Sigma^n,g,\rho)\vol(\Sigma^n,g)^{\frac{2 - n}{n}}\left(\int_{\partial \Sigma^n}\rho\,dv_g\right)
$$
in $\mC\times C^\infty(\partial \Sigma^n)$.
Then there exists a collection $u_1,\ldots, u_m$ of $\sigma_k(\Sigma^n,g,\rho)$-eigenfunctions such that 
\begin{itemize}
    \item[1)] $\sum_{i=1}^m u_i^2 = \frac{1}{\sigma_k(g,\rho) \vol(\del \Sigma, \rho)}$ on $\partial \Sigma^n$;
    \item[2)] $\sum_{i=1}^m |d u_i|^2_g = \frac{1}{\vol(\Sigma, g)}$ on $\Sigma^n$.
\end{itemize}
Conversely, if there exists a collection of $\sigma_k(\Sigma^n,g,\rho)$-eigenfunctions satisfying $1)-2)$ and additionally $\sigma_k(\Sigma^n,g,\rho)>\sigma_{k-1}(\Sigma^n,g,\rho)$ or $\sigma_k(\Sigma^n,g,\rho)<\sigma_{k+1}(\Sigma^n,g,\rho)$, then $g$ is extremal for the functional $\bar\sigma_k(\Sigma^n,g,\rho)$ in $\mC\times C^\infty(\partial \Sigma^n)$. 
\end{theorem}
\begin{remark}
The formulation of this theorem is almost identical to that of Theorem~\ref{confSteklov:thm}. In comparison with Theorem~\ref{Laplacianrho:thm}, note that the density is not necessarily constant. 
\end{remark}

\begin{proof}
    For smooth families $g(t) = e^{f(t)}g$ and $\rho(t)$ of metrics and density on $\Sigma$, 
    the eigenvalue $\sigma_k(g(t), \rho(t))$ is Lipschitz and its derivative, when it exists, is given by
    \begin{align*}
        \dot\sigma_k(g(t), \rho(t)) = Q_{(\dot{f}, \dot{f}, \dot{\rho})}(u),
    \end{align*}
    where $u \in E_k(g(t), \rho(t)), \|u\|^2_{L^2(\del M, \rho(t))} = \int_{\del \Sigma} u^2 \rho(t) dv_{g(t)} = 1$
    and 
    \begin{align*}
        Q_{(\phi,\psi,\eta)}(u) = - \int_\Sigma \frac{2 - n}{2}|\nabla u|^2 \phi
        - \sigma_k(g(t), \rho(t)) \int_{\del \Sigma} u^2 \left( \frac{n - 1}{2} \psi \rho + \eta \right).
    \end{align*}
    Let $\mathcal{H} = L^2(\Sigma) \times L^2(\del \Sigma, \rho) \times L^2(\del \Sigma)$ with
    the induced inner product. For $(g,\rho)$ conformally extremal, we again have
    that for any $(\phi, \psi, \eta) \in \mathcal{H}$ such that 
    $$\Big\langle(\phi,\psi,\eta), \left(\frac{2-n}{\vol(\Sigma, g)}, \frac{n-1}{\vol(\del \Sigma, \rho)},
    \frac{2}{\vol(\del \Sigma, g)}\right)\Big\rangle_{\mathcal{H}} = 0,$$
    there exists $u \in E_k(g, \rho)$ with $Q_{(\phi,\psi,\eta)}(u) = 0$. Then using
    Hahn-Banach's theorem we conclude that  $\left(\frac{2 - n}{\vol(\Sigma)}, \frac{n-1}{\vol(\del \Sigma, \rho)},
    \frac{2}{\vol(\del \Sigma, \rho)}\right)$ is in the convex hull of $\left\{\left(\frac{2 - n}{2} |\nabla u|^2,
    \frac{n-1}{2} \sigma_k u^2, \sigma_k u^2\right), u \in E_k(g,\rho)\right\}$.
    This implies the existence of $u_1, \dots, u_m \in E_k(g,\rho)$ such that
    \begin{align*}
        \begin{cases}
            \sum_{j = 1}^m |\nabla u_j|^2 = \frac{2}{\vol(\Sigma, g)} \quad \text{in } \Sigma \\
            \sum_{j = 1}^m \sigma_k u^2 = \frac{2}{\vol(\del \Sigma, \rho)} \quad \text{on } \del \Sigma. \\
        \end{cases}
    \end{align*}

    The proof of the converse is similar to the no-density Steklov's one.
\end{proof}

\subsection{Extremality conditions in the space of all metrics}

\begin{theorem}
\label{LaplaceR:thm}
Let $M^n$ be an $n$-dimensional closed manifold, $n\geqslant 3$. Suppose that the metric $g\in \mC$ is extremal for the functional 
$$
\bar\lambda_k(M^n,g) = \lambda_k(M^n,g) \vol(M^n,g)^\frac{2}{n}
$$
 in $\mR$. Then there exists a collection $u_1,\ldots, u_m$ of $\lambda_k(M^n,g)$-eigenfunctions such that 
\begin{itemize}
    \item[1)] $\sum_{i=1}^m u_i^2 = \frac{n}{\lambda_k(g)}$
    \item[2)] $\sum_{i=1}^m d u_i\otimes du_i = g$
\end{itemize}
Conversely, if there exists a collection $\lambda_k(M^n,g)$-eigenfunctions satisfying $1),2)$ and in addition $\lambda_k(M^n,g)>\lambda_{k-1}(M^n,g)$ or $\lambda_k(M^n,g)<\lambda_{k+1}(M^n,g)$, then $g$ is extremal for the functional $\bar\lambda_k(M^n,g)$ in $\mR$.
\end{theorem}
\begin{proof}
    For a smooth family $g(t)$ with $g(0) = g$ and $\frac{d}{dt} g(t) = h(t)$, the eigenvalue 
    $\lambda_k(g(t))$ is Lipschitz and its derivative, when it exists, is given by
    \begin{align*}
        \dot\lambda_k(g(t)) = Q_h(u) := - \int_M \ip{du \otimes du - \frac{1}{2} |\nabla u|^2 g(t) + \frac{1}{2} \lambda_k(g(t)) u^2 g(t), h} dv_{g(t)}
    \end{align*}
    where $u \in E_k(g(t)), \|u\|_{L^2(M,g(t))} = 1$.
    For $g$ extremal, we have that for any $h \in L^2(S^2(M))$ with $\int_M \ip{h,g} dv_g = 0$, there exists
    $u \in E_k(g)$ such that $Q_h(u) = 0$. Using Hahn-Banach's theorem, it then follows that
    there exists $u_1, \dots, u_m \in E_k(g)$ such that
    \begin{align*}
        g = \sum_{j = 1}^m du_j \otimes du_j - \frac{1}{2} |\nabla u_j|^2 g + \frac{1}{2} \lambda_k(g) u_j^2 g
    \end{align*}
    Taking the traceless part, we obtain $\sum du_j \otimes du_j = \frac{1}{n} |\nabla u_j|^2 g$, while
    taking the trace, we obtain the equation \ref{eq:laplace-conf-res} (up to a factor $n$) hence
    $\sum u_j^2 = \frac{n}{\lambda_k(g)}$ and $\sum |\nabla u_j|^2 = n$. Thus $\sum du_j \otimes du_j = g$. 

    The converse's proof is once again similar and is omitted. 
\end{proof}

\begin{theorem}
\label{SteklovR:thm}
Let $\Sigma^n$ be an $n$-dimensional connected compact manifold with non-empty boundary, $n\geqslant 3$. Suppose that the pair $(g,\rho)$ is extremal for the functional 
$$
\bar\sigma_k(\Sigma^n,g,\rho) = \sigma_k(\Sigma^n,g,\rho)\vol(N,g)^{\frac{2 - n}{n}}\left(\int_{\partial N}\rho\,dv_g\right)
$$
in $\mR\times C^\infty(\Sigma^n)$.
Then $\rho$ is a constant function $\rho \equiv \rho_0$ and there exists a collection $u_1,\ldots, u_m$ of $\sigma_k(\Sigma^n,g,\rho)$-eigenfunctions such that 
\begin{itemize}
    \item[1)] $\sum_{i=1}^m u_i^2 = \frac{1}{\sigma_k \vol(\del \Sigma, \rho)}$ on $\partial \Sigma^n$;
    \item[2)] $\sum_{i=1}^m d u_i\otimes du_i = \frac{g}{\vol(\Sigma,g)}$ on $\Sigma^n$.
\end{itemize}
Conversely, if $\rho\equiv\rho_0$ and there is a collection of $\sigma_k(\Sigma^n,g,\rho)$-eigenfunctions satisfying $1),2)$ and in addition $\sigma_k(\Sigma^n,g,\rho)>\sigma_{k-1}(\Sigma^n,g,\rho)$ or $\sigma_k(\Sigma^n,g,\rho)<\sigma_{k+1}(\Sigma^n,g,\rho)$, then $g$ is extremal for the functional $\bar\sigma_k(\Sigma^n,g,\rho)$ in $\mC\times C^\infty(\partial \Sigma^n)$. 
\end{theorem}
\begin{remark}
In comparison with Theorem~\ref{Steklovrho:thm}, we see that extending the deformation space to $\mR$ forces the density to be constant. 
\end{remark}
\begin{proof}
    For smooth families $g(t), \rho(t)$ with $g(0) = g, \frac{d}{dt}g(t) = h(t)$
    and $\rho(0) = \rho, \frac{d}{dt}\rho(t) = \eta(t)$, the eigenvalue 
    $\sigma_k(g(t), \rho(t))$ is Lipschitz and when its derivative exists, it is given by
    \begin{align*}
        \dot\sigma_k(g(t), \rho(t))
        &= - \int_\Sigma \ip{h, du \otimes du - \frac{1}{2} |\nabla u|^2 g}
        - \sigma_k \int_{\del \Sigma} u^2 (\eta + \frac{\rho}{2} \ip{h|_{\del \Sigma}, g|_{\del \Sigma}}, \\
        &=: Q_{(h,\ip{h|_{\del \Sigma}, g|_{\del \Sigma}}, \eta)}(u)
    \end{align*}
    where $u \in E_k(g(t), \rho(t)), \|u\|_{L^2(\del \Sigma, \rho)} = 1$. 

    Let $\mathcal{H} = L^2(S^2(\Sigma)) \times L^2(\del \Sigma, \rho) \times L^2(\del \Sigma, g)$
    with the induced inner product from the $L^2$ spaces.  For $(g,\rho)$ extremal, we still have that for
    any $(h, f, \eta) \in \mathcal{H}$ such that 
    \begin{align*}
        \left\langle (h,f,\eta), \left(\frac{2-n}{n} \frac{1}{\vol(\Sigma,g)}, \frac{1}{\vol(\del \Sigma, \rho) },
        \frac{2}{\vol(\del \Sigma, \rho)}\right) \right\rangle= 0,
    \end{align*}
    there exists $u \in E_K(g,\rho)$ such that $Q_{(h,f,\eta)}(u) = 0$. Then the Hahn-Banach argument
    gives that there exists $u_1, \dots, u_m \in E_k(g,\rho)$ such that
    \begin{align*}
        \begin{cases}
            \frac{2 - n}{n} \frac{g}{\vol(\Sigma,g)} = \sum_{j = 1}^m \left( du_j \otimes du_j - \frac{1}{2} |\nabla u_j|^2 g \right), \\
            \frac{2}{\sigma_k \vol(\del \Sigma, \rho)} = \sum_{j = 1}^m u_j^2.
        \end{cases}
    \end{align*}
    Taking the trace of the first equation yields $\sum |\nabla u_j|^2 = \frac{2}{\vol(\Sigma,g)}$ which
    when used in the first equation gives $\sum du_j \otimes du_j = \frac{2g}{\vol(\Sigma,g)}$. By rescaling
    the $u_j$ we get then get the desired equalities. 

    Using the two equalities, we obtain for the density $\rho$:
    \begin{align*}
        \frac{1}{\vol(\Sigma, g)} &= \frac{g(\nu,\nu)}{\vol(\Sigma, g)} \\
          &= \sum_{j = 1}^m du_j \otimes du_j (\nu,\nu) = \sum_{j = 1}^m |\del_\nu u_j|^2 \\
          &= \sigma_k^2 \rho^2 \sum_{j = 1}^m u_j^2  
          = \frac{\sigma_k \rho^2}{\vol(\del \Sigma,\rho)}
    \end{align*}
    hence $\rho$ is constant on $\del \Sigma$.

    The proof of the converse is once again the same and is thus omitted.
\end{proof}

\section{Geometric extremality conditions} \label{geometricextremal:sec}

In this section we complete the proofs of geometric characterizations of extremal metrics.

\subsection{Proof of Theorem~\ref{Laplacianconf:thm}}
\label{1:sec}
 Let $g$ be a $\bar\lambda_k$-conformally extremal and $\lambda_k=\bar\lambda_k(M^n,g)$. Then by Theorem~\ref{Laplacianconfalg:thm} there is a collection $\{u_1,\ldots,u_m\}$ such that $\Delta_g u_i=\lambda_k u_i$ and
\begin{equation}
\label{cond1:eq}
\sum_{i=1}^m u_i^2 = \frac{1}{\lambda_k}\qquad \sum_{i=1}^m |du_i|_g^2=1.
\end{equation}
Let $\Phi\colon M^n\to \mathbb{S}^{m-1}$ be defined by $\Phi=\sqrt{\lambda_k}(u_1,\ldots, u_m)$. Then $|d\Phi|_g^2=\lambda_k$ and, therefore,
$$
\delta_g(|d\Phi|_g^{n-2}d\Phi) = \lambda_k^{\frac{n-2}{2}}\Delta_g\Phi = \lambda_k^\frac{n}{2}\Phi = |d\Phi|_g^n\Phi
$$
i.e. $\Phi\colon (M,g)\to\mathbb{S}^{m-1}$ is a non-degenerate $n$-harmonic map by~\eqref{n-harmonic:eq}. Finally, $g_\Phi = \frac{1}{n}|d\Phi|_g^2 g = \frac{\lambda_k}{n}g$, i.e. $g_\Phi$ is a constant multiple of $g$.

Conversely, If $\Phi\colon (M,g)\to\mathbb{S}^{m-1}$ is an $n$-harmonic map, then for $k=\ind_S(\Phi)$, one has $\lambda_k=\lambda_{k}(M^n,g_\Phi) = n$ and the components of $\frac{1}{\sqrt{\lambda_k}}\Phi$ are $\lambda_k$-eigenfunctions of $\Delta_{g_\Phi}$ satisfying~\eqref{cond1:eq}. Since by the definition of $\ind_S(\Phi)$, one has $\lambda_{k-1}<\lambda_k$, the conclusion follows from the second part of Theorem~\ref{Laplacianconf:thm}.

\subsection{Proof of Theorem~\ref{Steklovconf:thm}} First, let us show that for a non-degenerate free boundary $n$-harmonic maps $\hat\Psi$ the corresponding density $\rho_{\hat\Psi}$ is positive.

\begin{lemma} 
\label{density:lemma}
Let $\hat\Psi\colon(\Sigma^n,g)\to\mathbb{B}^m$ be a non-degenerate free boundary $n$-harmonic map. Then $\rho_{\hat\Psi}>0$.
\end{lemma}
\begin{proof}
Let $h=g_{\hat\Psi}$, then one has $\Delta_h\hat\Psi = 0$ in the interior of $\Sigma^n$. Then $\Delta_h(|\hat\Psi|^2) = -2|d\hat\Psi|_h^2 = -2n<0$, i.e. $|\hat\Psi|^2$ is a subharmonic function, which attains its maximum at every point of $\partial\Sigma^n$. Thus, by the boundary point maximum principle $\partial_{\nu_h}|\hat\Psi|^2>0$, where $\nu_h$ is the outer unit normal vector w.r.t. $h$. At the same time,
$$
\partial_{\nu_h}|\hat\Psi|^2 = 2\hat\Psi\cdot\partial_{\nu_h}\hat\Psi = 2|\partial_{\nu_h}\hat\Psi| = 2\rho_{\hat\Psi}.
$$
\end{proof}

Let $(g,\rho)$ be an extremal pair. Let $\sigma_k=\sigma_k(\Sigma^n,g,\rho)$, then by Theorem~\ref{Steklovrho:thm} there is a collection $(u_1,\ldots,u_m)$ of $\sigma_k$-eigenfunctions such that 
\begin{equation}
\label{cond2:eq}
\sum_{i=1}^m u_i^2 = \frac{1}{\sigma_k\|\rho\|_{L^1}}\text{ on }\partial\Sigma^n,\qquad \sum_{i=1}^m|du_i|^2_g = \frac{1}{\vol(\Sigma^n,g)}\text{ on }\Sigma^n
\end{equation}
Set $\hat\Psi\colon\Sigma^n\to\mathbb{R}^m$ be given by $\sqrt{\sigma_k\|\rho\|_{L^1}}(u_1,\ldots, u_m)$. Since $\Delta_g\hat\Psi=0$ one has $\Delta_g|\hat\Psi|^2_g = -2|d\hat\Psi|_g^2<0$. Thus, by maximum principle $|\hat\Psi|^2_g$ achieves it's maximum on the boundary, therefore, $\hat\Psi\colon\Sigma^n\to\mathbb{B}^m$. 
Furthermore, 
$$
|d\hat\Psi^n|_g^2 = \frac{\sigma_k\|\rho\|_{L^1}}{\vol(\Sigma^n,g)},\quad \partial_{\nu_g}\hat\Psi = \rho\hat\Psi\parallel\hat\Psi.
$$

Since $|d\hat\Psi|^2_g$ is constant, one has
$$
\delta_g(|d\hat\Psi|_g^{n-2}d\hat\Psi) = |d\hat\Psi|_g^{n-2}\Delta_g\hat\Psi=0,
$$
therefore, $\hat\Psi$ is a non-degenerate free boundary $n$-harmonic map. Moreover, $g_{\hat\Psi} = \frac{1}{n}|d\hat\Psi|_g^2 g = \frac{\sigma_k\|\rho\|_{L^1}}{n\vol(\Sigma^n,g)}g$ and $\rho_{\hat\Psi} = \frac{\sqrt{n}|\partial_{\nu_g}\hat\Psi|}{|d\hat\Psi|_g} = \sqrt{\frac{n\vol(\Sigma^n,g)}{\sigma_k\|\rho\|_{L^1}}}\rho$, which completes the proof.

Conversely, let $\hat\Psi\colon(\Sigma^n,g)\to\mathbb{B}^m$ be a free boundary $n$-harmonic map. Let $k=\ind_S(\hat\Psi)$, then one has $\sigma_k = \sigma_k(\Sigma^n,g_{\hat\Psi},\rho_{\hat\Psi}) = 1$ and components of 
$\frac{1}{\sqrt{\sigma_k\|\rho_{\hat\Psi}\|_{L^1}}}\hat\Psi$ are $\sigma_k$-eigenfunctions satisfying 
$$
\sum_{i=1}^m u_i^2 = \frac{1}{\sigma_k\|\rho_{\hat\Psi}\|_{L^1}}\text{ on }\partial\Sigma^n,\qquad \sum_{i=1}^m|du_i|^2_g = \frac{1}{\sigma_k\|\rho_{\hat\Psi}\|_{L^1}}\text{ on }\Sigma^n.
$$
In order to apply the second part of Theorem~\ref{Steklovrho:thm} we need to show that $\sigma_k\|\rho_{\hat\Psi}\|_{L^1} = \vol(\Sigma^n,g_{\hat\Psi})$. To achieve this, note that $\Delta_{g_{\hat\Psi}}|\hat\Psi|^2 = -2|d\hat\Psi|_{g_{\hat\Psi}} = -2$. Thus, integrating $\frac{1}{2}\Delta_{g_{\hat\Psi}}|\hat\Psi|^2$ and applying Green's formula yields
$$
\vol(\Sigma^n,g_{\hat\Psi}) = \int_{\partial\Sigma^n}\sigma_k\rho_{\hat\Psi} = \sigma_k\|\rho_{\hat\Psi}\|_{L^1}.
$$

\subsection{Proof of Theorem~\ref{Laplace-extremal:thm}} 
\label{3:sec}
Let $g$ be a $\bar\lambda_k$-extremal metric. Then by Theorem~\ref{LaplaceR:thm} there exists a collection $(u_1,\ldots, u_m)$ of $\lambda_k(M^n,g)$ eigenfunctions satisfying
$$
\sum_{i=1}^m u_i^2 = \frac{n}{\lambda_k(M^n,g)},\qquad \sum_{i=1}^m du_i\otimes du_i = g.
$$
In particular, taking the trace of the second equality yields 
$$
\sum_{i=1}^m |du_i|_g^2 = n.
$$
Therefore, one can repeat the arguments of Section~\ref{1:sec} to conclude that $\Phi = \sqrt{\frac{\lambda_k(M^n,g)}{n}}(u_1,\ldots,u_m)$ is a non-degenerate $n$-harmonic map to the sphere $\mathbb{S}^{m-1}$. Additionally, one has that the pullback $\Phi^*g_{\mathbb{S}^{m-1}}$ is proportional to $g$ and, in particular, the map $\Phi$ is conformal. Then by~\cite[Corollary 4]{Takeuchi} $\Phi$ is a minimal immersion.

The proof of converse is identical to that in Section~\ref{1:sec}. The only new observation is that for a minimal immersion $\Phi$ one has $g_\Phi = \Phi^*g_{\mathbb{S}^{m-1}}$.

\subsection{Proof of Theorem~\ref{Steklov-extremal:thm}} The proof is identical to that in Section~\ref{3:sec}. The important point is that for any free boundary minimal immersion $\hat\Psi$ one has $\rho_{\hat\Psi} \equiv 1$.

\section{Extremal metrics on the annulus}
\label{Annuli:sec}
The aim of this section is to study the Steklov problem on the annulus and
prove Theorem \ref{annulus:thm} and Proposition \ref{annulus:prop}.

\subsection{Setup of the problem}
Let $\mathbb{A}_T = [0,T] \times S^1$ an annulus with metric $g = f(t)(dt^2 + d\theta^2)$.
By the uniformization theorem, any annulus with a rotationally symmetric metric is isometric to an annulus of this form for some $T$ which fixes its conformal class.
The Steklov problem on $\mathbb{A}_T$ is then
\begin{align} \label{eq:Steklov-annulus}
    \begin{cases}
        \Delta u = 0 & \text{in } \mathbb{A}_T \\
        \del_t u = - \sigma f(0) u & \text{on } \{0\} \times S^1 \\
        \del_t u = \sigma f(T) u & \text{on } \{T\} \times S^1
    \end{cases}.
\end{align}
Let $\rho_1 = f(0)$ and $\rho_2 = f(T)$. Then
simple calculations, using that the Laplacian is conformally invariant in 2-dimension, 
give that the weighted Steklov eigenfunctions are:
\begin{align*}
    u_1(t,\theta) &= -1 + \sigma^{(0)} \rho_1  t  \\
    u_2^{(n)}(t,\theta) &= \cosh(t) \cos(\theta) - \sigma^{(n)}_{\pm} \rho_1 \sinh(t) \cos(\theta) \\
    u_3^{(n)}(t,\theta) &= \cosh(t) \sin(\theta) - \sigma^{(n)}_{\pm} \rho_1 \sinh(t) \sin(\theta)
\end{align*}
where $\sigma^{(0)}$ and $\sigma^{(n)}_{\pm}$ are their corresponding eigenvalues, given by
\begin{align*}
    \sigma^{(0)} &= \frac{1}{T}\left(\frac{1}{\rho_1 } + \frac{1}{\rho_2 }\right) \\
    \sigma^{(n)}_{\pm}
                 &= \frac{n}{2} \left( \left(\frac{1}{\rho_1} + \frac{1}{\rho_2}\right) \coth(nT)
                     \pm \left( \left( \frac{1}{\rho_1} + \frac{1}{\rho_2}\right)^2 \coth(nT)^2
                 - \frac{4}{\rho_1 \rho_2} \right)^{1/2} \right).
\end{align*}
For our ends, we are only interested in the lowest (non zero) eigenvalue and focus on $\sigma^{(0)}$ and $\sigma^{(n)}_{-}$.
Using the identity $x - y = \frac{x^2 - y^2}{x + y}$ we see that
$\sigma^{(n)}_{-}$ is increasing with $n$. Hence we only need
to consider $\sigma^{(0)}$ and $\sigma^{(1)}_{-}$. To get an immersion in $\mathbb{R}^3$
by first eigenfunctions we need the first eigenvalue to have multiplicity 3.
But $\sigma^{(0)}$ and $\sigma^{(1)}_{-}$
have multiplicity $1$ and $2$ respectively. Hence we obtain the condition
$\sigma_1 = \sigma^{(0)} = \sigma^{(1)}_{-}$ which gives the equation
\begin{align*}
    \frac{1}{T}
    &= \frac{2 \rho_1 \rho_2}{(\rho_1 + \rho_2)^2} 
    \left(\coth(T) + \left(\coth(T)^2 - \frac{4 \rho_1 \rho_2}{(\rho_1 + \rho_2)^2}\right)^{1/2}\right)^{-1}
\end{align*}
The left hand side is decreasing from $+\infty$ at $T = 0$ to $0$ at $T = +\infty$ while
the right hand side is increasing on $\mathbb{R}^+$. Hence the equation has a unique solution $T_q$
where $q = \frac{\rho_1}{\rho_2}$ (the fact that the solution depends only on this ratio
can be seen be rewriting $\frac{4 \rho_1 \rho_2}{(\rho_1 + \rho_2)^2} = \frac{4 q}{(1+ q)^2}$).

We now fix $T_q$ and consider it as a function of $q$. To simplify notation we write $u_2^{(1)} = u_2, u_3^{(1)} = u_3$ and
$\sigma_1 = \sigma^{(0)} = \sigma^{(1)}_{-} = \frac{\rho_1 + \rho_2}{\rho_1 \rho_2 T_q}$.
We have the following results on $T_q$
\begin{lemma} \label{lem:Tq-facts}
    $T_q$ is decreasing for $q \in (0,1]$ and increasing for $q \in [1, +\infty)$, with a minimum
    at $q = 1$. Furthermore, for all $q > 0$, 
    \begin{align*}
        \frac{(1+q)^2}{2q} < \frac{(1+q)^2}{2q} \coth(T_q) < T_q < \frac{(1+q)^2}{q}
    \end{align*}
    and
    \begin{align*}
        T_q > q + \frac{1}{q}.
    \end{align*}
\end{lemma}
\begin{proof}
    Due to the symmetry of the problem when interchanging $\rho_1$ and $\rho_2$, we have
    $T_q = T_{1/q}$. To prove the first part of the lemma, it is then enough to show $T_q$
    decreasing for $q \in (0,1]$. Suppose on the contrary that $T_q$ is increasing for
    some $q \in (0,1]$. Then we see in the equation for $T_q$,
    \begin{align} \label{eq:Tq}
        T_q = \frac{(1+q)^2}{2q} \left(\coth(T_q) + \left(\coth(T_q) - \frac{4q}{(1+q)^2}\right)^{1/2}\right)
    \end{align}
    that the right hand side is decreasing, a contradiction. 

    The first two lower bounds are directly obtained from the equation (\ref{eq:Tq}) for $T_q$
    \begin{align*}
        T_q &= \frac{(1+q)^2}{2q} \left(\coth(T_q) + \left(\coth(T_q)^2 - \frac{4q}{(1+q)^2}\right)^{1/2}\right) \\
            &> \frac{(1+q)^2}{2q} \coth(T_q) > \frac{(1+q)^2}{2q}
    \end{align*}
    For the last lower bound we have
    \begin{align*}
        \frac{(1+q)^2}{q} \coth(T_q) - T_q &= \frac{2}{\coth(T_q) + \left(\coth(T_q)^2 - \frac{4q}{(1+q)^2}\right)^{1/2}} < 2.
    \end{align*}
    So
    \begin{align*}
        T_q > \frac{(1+q)^2}{q} \coth(T_q) - 2 > \frac{(1+q)^2}{q} - 2 = q + \frac{1}{q}.
    \end{align*}

    More work is required for the upper bound. From the equation for $T_q$, the upper bound is equivalent
    to the inequality
    \begin{align*}
        \coth(T_q) + \left(\coth(T_q)^2 - \frac{4q}{(1+q)^2}\right)^{1/2} < 2,
    \end{align*}
    itself equivalent to
    \begin{align*}
        \coth(T_q) < 1 + \frac{q}{(1+q)^2}.
    \end{align*}
    From the lower bound on $T_q$ and since $\coth$ is decreasing, $\coth(T_q) < \coth\left(\frac{(1+q)^2}{2q}\right)$. 
    Thus it suffices to show $\coth\left(\frac{(1+q)^2}{2q}\right) < 1 + \frac{q}{(1+q)^2}$,
    i.e. $\coth(x) < 1 + \frac{1}{2x}$ for $x \geqslant 2$ by setting $x = \frac{(1+q)^2}{2q}$. Using that
    $\coth(x) - 1 = \frac{2}{e^{2x} - 1}$, proving this last inequality is equivalent to proving
    \begin{align*}
        4x < e^{2x} - 1 \quad \text{for } x \geqslant 2,
    \end{align*}
    which is clear since $8 < e^4 - 1 \approx 53$ and $e^{2x}$ increases faster than $4x$. 
\end{proof}

\subsection{Proof of Theorem~\ref{annulus:thm}}
Since $T \geqslant T_1$ and $T_q \to \infty$ when $q \to \infty$, we can find a $f : [0,T] \to \mathbb{R}^+$
such that $T = T_q$ with $q = \frac{f(0)}{f(T)}$.

From the eigenfunctions $u_1, u_2, u_3$ found above, 
we want to construct a harmonic map $\hat\Psi_T : \mathbb{A}_T \to B^3$ such that $\hat\Psi_T(\del \mathbb{A}_T) \subset S^2$.
Let $c_1, c_2 \in \mathbb{R}^+$ and let $\hat\Psi_T = (c_1 u_1, c_2 u_2, c_2 u_2)$ then the boundary
condition $\hat\Psi_T(\del \mathbb{A}_T) \subset S^2$ implies
\begin{align*}
    \begin{cases}
        1 &= c_1^2 + c_2^2 \\
        1 &= c_1^2 a(q)^2 + c_2^2 b(q)^2.
    \end{cases}
\end{align*}
where
\begin{align*}
    a(q) &= -1 +  \sigma_1 \rho_1 T_q = q\\
    b(q) &= \cosh(T_q) - \sigma_1 \rho_1 \sinh(T_q).
\end{align*}
Clearly a solution $(c_1, c_2)$ exists if and only if
$(a(q)^2 = b(q)^2 = 1)$ or $(a(q)^2 < 1, b(q)^2 > 1)$ or $(a(q)^2 > 1, b(q)^2 < 1)$. The case $a(1)^2 = b(1)^2 = 1$
gives the critical catenoid and the map $\hat\Psi_{T_1}$, proving (1). 

We now show that a solution exists when $q \neq 1$. Since the problem does not change when switching $\rho_1
\leftrightarrow \rho_2$, we can consider without loss of generality
only the case $0< q < 1$ and it suffices to show $b(q)^2 > 1$. Knowing that $b(1) = 1$, we will
show that $b$ is decreasing on $(0,1)$. 

The derivative of $b$ is 
\begin{align*}
    b'(q) 
           &= \frac{\sinh(T_q)}{T_q^2} \Big(T_q' \left(q + 1 + T_q^2 - (1+q) T_q \coth(T_q)\right)
           - T_q \Big)
\end{align*}
and since $T_q > 0$ and by lemma \ref{lem:Tq-facts}, $T_q$ is decreasing for $q \in (0,1)$, it suffices to show that 
\begin{align*}
    0< T_q - (1+q) \coth(T_q).
\end{align*}
By lemma \ref{lem:Tq-facts}, $T_q > \frac{(1+q)^2}{2q} \coth(T_q)$ hence
\begin{align*}
    T_q - (1+q) \coth(T_q) > (1 + q)T_q \coth(T_q) \left(\frac{1+q}{2q} - 1\right) > 0
\end{align*}
for $q \in (0,1)$. 
This proves that for any $T \geqslant T_1$, there exists a rotationally symmetric free boundary harmonic map
$\hat\Psi_T : \mathbb{A}_T \to \mathbb{B}^3$ of spectral index 1. 

Examples of the surface $\hat\Psi_{T}(\mathbb{A}_T)$ are shown in figure \ref{fig:examples}.
As can be seen from those figures and direct calculations, $\hat\Psi_T(\mathbb{A}_T)$ is a section of
a stretched catenoid. When $T \to \infty$, this stretched catenoid collapses to a disk with a segment 
orthogonal to it.

\begin{figure}
    \centering
    \begin{subfigure}[h]{0.4\textwidth}
        \includegraphics[width=\textwidth]{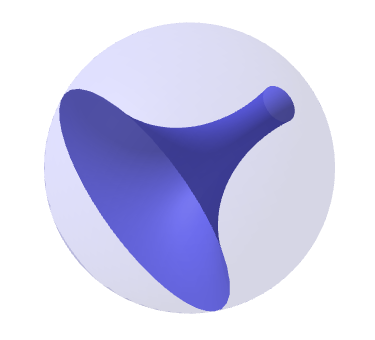}
        \caption{$q = 1/2$}
    \end{subfigure}
    \begin{subfigure}[h]{0.4\textwidth}
        \includegraphics[width=\textwidth]{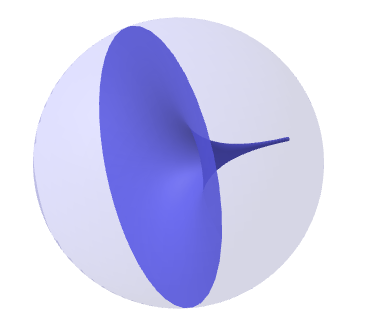}
        \caption{$q = 1/4$}
    \end{subfigure}
    \caption{Some examples of the surface $\hat\Psi_{T_q}(\mathbb{A}_{T_q})$ for different values of $q = \frac{\rho_1}{\rho_2}$.}
    \label{fig:examples}
\end{figure}

\emph{Proof of (2).} The fact that the pair $(g_{\hat\Psi_T}, \rho_{\hat\Psi_T})$ is rotationally symmetric
is clear by construction. The density $\rho_{\hat\Psi_T} = \frac{\sqrt{n}\sigma_1}{|d\hat\Psi_T|_g}$
so we prove that for $T \geqslant T_2 \approx 3.04$, $|d\hat\Psi_T|_g$ depends on the boundary component
of $\mathbb{A}_T$. Calculating for $t_1 = 0, t_2 = T$,
\begin{align*}
    |d\hat\Psi_T|_g^2(t_i,\theta) 
    = \sigma_1^2 + \frac{c_2^2}{\rho_i^2} (\cosh(t_j) - \sigma_1 \rho_1 \sinh(t_j))^2.
\end{align*}
Remark that for $t_1 = 0$ the second term is $\frac{c_2^2}{\rho_1^2}$ while for $t_2 = T =:T_q$
it is $\frac{c_2^2 b(q)^2}{\rho_2^2}$. Hence we will prove that $b(q) \neq \frac{\rho_2}{\rho_1} = \frac{1}{q}$
for $T_q \geqslant T_2$, i.e for $q \leqslant 2$ and for $q \geqslant 2$. 
As discussed previously, it is enough to prove the case $q \geqslant 2$.
Rewriting $b(q)$ as
\begin{align*}
    b(q) &= \cosh(T_q) - \frac{1}{2}(q+1)\left(\coth(T_q) - \left(\coth(T_q)^2 - \frac{4q}{(1+q)^2} \right)^{\frac{1}{2}}\right) \sinh(T_q) \\
         &= \frac{2q}{\Big((1+q)^2 \cosh(T_q)^2 - 4q \sinh(T_q)^2\Big)^{1/2} - (1-q)\cosh(T_q)} \\
         &= \frac{2q}{\frac{2q}{1+q}T_q \sinh(T_q) - (1+q) \cosh(T_q) - (1-q)\cosh(T_q)} \\
         &= \frac{1}{\frac{T_q}{1+q}\sinh(T_q) - \frac{1}{q} \cosh(T_q)},
\end{align*}
where the second to last equality comes from the equation for $T_q$, it is enough to show
$\frac{T_q}{1+q} \sinh(T_q) - \frac{1}{q}\cosh(T_q) \neq q$. Using lemma \ref{lem:Tq-facts}, we have the lower bound
\begin{align*}
     \frac{T_q}{1+q} \sinh(T_q) &- \frac{1}{q}\cosh(T_q) 
        \geqslant \\ 
        &\geqslant\frac{1}{1+q} \left(\frac{(1+q)^2}{q} \coth(T_q) - 2\right) \sinh(T_q) - \frac{1}{q} \cosh(T_q) \\
        &= \left(1 - \frac{2}{1+q}\right) \cosh(T_q) + \frac{2}{1+q} e^{-T_q} \\
        &\geqslant \frac{q-1}{q+1} \cosh(T_q) \geqslant \frac{q-1}{q+1} \cosh\left(q+ \frac{1}{q}\right).
\end{align*}
For $q = 2$, we have $\frac{q-1}{q+1} \cosh\left(q + \frac{1}{q}\right) \approx 3.1 > 2 = q$ and
since $\frac{q-1}{q+1} \cosh(q + 1/q)$ grows faster than $q$, this proves part (2).

\emph{Proof of (3).} By construction, $\bar\sigma_1 (\mathbb{A}_T, g_{\hat\Psi_T}, \rho_{\hat\Psi_T}) 
=  2E(\hat\Psi_T)$ and
\begin{align*}
    \bar\sigma_1(\mathbb{A}_T, g_{\hat\Psi_T}, \rho_{\hat\Psi_T})
    = \frac{1}{T_q} \left(\frac{1}{\rho_1} + \frac{1}{\rho_2}\right) 2\pi (\rho_1 + \rho_2)
    = \frac{2\pi (1+q)^2}{qT_q}
\end{align*}
By lemma \ref{lem:Tq-facts}, $q + \frac{1}{q} < T_q < \frac{(1+q)^2}{q}$ hence
\begin{align*}
    2\pi < \bar\sigma_1(\mathbb{A}_T, g_{\hat\Psi_T}, \rho_{\hat\Psi_T}) < 2\pi + \frac{4\pi q}{q^2 + 1}.
\end{align*}
This provides the desired lower bound while the limit is obtained by noting that by lemma \ref{lem:Tq-facts},
when $T \to \infty$, $q$ goes to $0$ or $\infty$. 

\subsection{Proof of Proposition~\ref{annulus:prop}}
From lemma \ref{lem:Tq-facts}, $T_q$ achieves its minimum at $q=1$ hence for the conformal class on the annulus
with $T < T_1$, one cannot have a rotationally symmetric $\bar\sigma_1$-conformally extremal metric given by a map to $B^3$.
The only possible rotationally symmetric $\bar\sigma_1$-conformally extremal metric comes instead from a map $\hat\Psi_T : \mathbb{A}_T \to \mathbb{B}^2$. 
From our previous discussion, since $T < T_1$, the components of $\hat\Psi_T$ are first eigenfunctions
$u_i(t,\theta) = c_i (\cosh(t) - \sigma_1 \rho_1' \sinh(t)) S_i(\theta)$ with $S_1 = \cos, S_2 = \sin$.
From the condition that $\hat\Psi_T(\del \mathbb{A}_T) \subset S^1$, we have $c_1^2 = c_2^2 = 1$ and we
must satisfy
\begin{align*}
    1 &= (\cosh(T) - \sigma_1 \rho_1 \sinh(T))^2
\end{align*}
But
\begin{align*}
    \cosh(T) - \sigma_1 \rho_1 \sinh(T)
    &= \frac{1}{2} (1- q) \cosh(T)  \\
    &+ \frac{1}{2}(1+q) \Big(\cosh(T)^2 - \frac{4q}{(1+q)^2}\sinh(T)^2\Big)^{1/2}
\end{align*}
so taking $f : [0,T] \to \mathbb{R}^+$ with $f(0) =: \rho_1 = 1 = \rho_2 := f(T)$,
gives $q = 1$ and the condition is satisfied.

The normalized eigenvalue is
\begin{align*}
    \sigma(\mathbb{A}_T, g_{\hat\Psi_T}, \rho_{\hat\Psi_T}) 
    &= 4\pi \left(\coth(T) - \left( \coth(T)^2- 1\right)^{1/2}\right) \\
    &= 4\pi \tanh(T/2).
\end{align*}
Hence for all $T \leqslant \tilde{T}$ where $\tilde{T} \approx 1.10$ is the solution of $2\tanh(T/2) = 1$,
$\sigma(\mathbb{A}_T, g_{\hat\Psi_T}, \rho_{\hat\Psi_T}) \leqslant 2\pi$. But it was shown
in \cite{MP} that the supremum of the normalized eigenvalue over a conformal class must be $> 2\pi$,
hence $(g_{\hat\Psi_T}, \rho_{\hat\Psi_T})$ is not a maximal pair.

\end{document}